\documentclass{amsart}
\usepackage{amsmath}
\usepackage{amssymb}
\usepackage{amsthm}
\usepackage[msc-links,abbrev]{amsrefs}
\usepackage{hyperref}
\usepackage[noabbrev,capitalize]{cleveref}
\usepackage{mathrsfs}
\usepackage{mathtools}
\usepackage{slashed}
\usepackage{tikz-cd}
\usepackage{enumitem}

\setenumerate{label=(\roman*)}

\DeclareMathOperator{\im}{im}
\DeclareMathOperator{\tr}{tr}

\DeclareMathOperator{\dvol}{dV}

\DeclareMathOperator{\Ric}{Ric}
\DeclareMathOperator{\Rm}{Rm}

\newcommand{\bg}{\boldsymbol{g}}

\newcommand{\defn}[1]{{\boldmath\bfseries#1}}

\newcommand{\cg}{\widetilde{g}}

\newcommand{\cu}{\widetilde{u}}
\newcommand{\cv}{\widetilde{v}}
\newcommand{\cw}{\widetilde{w}}
\newcommand{\cx}{\widetilde{x}}

\newcommand{\cD}{\widetilde{D}}

\newcommand{\cI}{\widetilde{I}}

\newcommand{\cQ}{\widetilde{Q}}

\newcommand{\cX}{\widetilde{X}}

\newcommand{\cdelta}{\widetilde{\delta}}
\newcommand{\cDelta}{\widetilde{\Delta}}

\newcommand{\cmE}{\widetilde{\mathcal{E}}}

\newcommand{\cmG}{\widetilde{\mathcal{G}}}

\newcommand{\cmU}{\widetilde{\mathcal{U}}}

\newcommand{\hg}{\widehat{g}}

\newcommand{\mE}{\mathcal{E}}

\newcommand{\mG}{\mathcal{G}}

\newcommand{\kD}{\mathfrak{D}}

\newcommand{\kc}{\mathfrak{c}}

\newcommand{\bN}{\mathbb{N}}

\newcommand{\bR}{\mathbb{R}}

\newcommand{\bZ}{\mathbb{Z}}

\newcommand{\sF}{\mathscr{F}}

\newcommand{\sI}{\mathscr{I}}

\newcommand{\lv}{\lvert}
\newcommand{\rv}{\rvert}

\def\sideremark#1{\ifvmode\leavevmode\fi\vadjust{\vbox to0pt{\vss
 \hbox to 0pt{\hskip\hsize\hskip1em
 \vbox{\hsize3cm\tiny\raggedright\pretolerance10000
 \noindent #1\hfill}\hss}\vbox to8pt{\vfil}\vss}}}

\newcommand{\suchthat}{\mathrel{}:\mathrel{}}

\newtheorem{theorem}{Theorem}[section]
\newtheorem{proposition}[theorem]{Proposition}
\newtheorem{lemma}[theorem]{Lemma}

\theoremstyle{definition}

\theoremstyle{remark}
\newtheorem{remark}[theorem]{Remark}

\numberwithin{equation}{section}

\begin{document}

\title[Conformally covariant tridifferential operators]{Conformally covariant tridifferential operators via the Fefferman--Graham ambient space}
\author{Jeffrey S.\ Case}
\address{Department of Mathematics \\ Penn State University \\ University Park, PA 16802 \\ USA}
\email{jscase@psu.edu}
\author{Opal Cieslak}
\address{Department of Mathematics, Statistics, and Computer Science \\
851 S.\ Morgan Street \\ 322 Science and Engineering Offices (MC 249) \\ Chicago, IL 60607-7045 \\ USA}
\email{jcies4@uic.edu}
\begin{abstract}
    We construct a large family of conformally covariant tridifferential operators as tangential operators in the Fefferman--Graham ambient space.
    Our construction is analogous to the linear and bilinear constructions of Graham--Jenne--Mason--Sparling and Case--Lin--Yuan, respectively.
    We also show that the symmetrization of our ambient operators are formally self-adjoint when acting on densities of the correct weight.
\end{abstract}

\maketitle

\section{Introduction}
\label{sec:intro}

A landmark result of Graham, Jenne, Mason, and Sparling~\cite{GJMS1992} established the existence of conformally covariant, (linear) differential operators, now called GJMS operators, with leading-order term a power of the Laplacian.
Their construction is computationally simple:\ 
the GJMS operator $P_{2k}$ with leading-order term $\Delta^k$ descends from the Laplacian $\cDelta^k$ in the Fefferman--Graham ambient space, as follows from a commutator formula using an $\mathfrak{sl}_2$-triple.
Analogous constructions can also be made on sections of different bundles~\cites{BransonGover2005,Matsumoto2013}, on submanifolds of conformal manifolds~\cite{CaseGrahamKuo2023}, and on CR manifolds~\cite{GoverGraham2005}.

This paper is a contribution to the problem of systematically constructing and classifying \emph{nonlinear} conformally invariant operators.
There are two significant previous works in this direction.
First, Alexakis~\cite{Alexakis2003} classified such operators in odd dimensions via harmonic homogeneous extensions, modeled on the classification of scalar Riemannian invariants by Bailey, Eastwood, and Graham~\cite{BaileyEastwoodGraham1994}.
Second, Case, Lin, and Yuan~\cite{CaseLinYuan2018b} proved a correspondence between formally self-adjoint, conformally invariant, multilinear differential operators and scalar Riemannian invariants that are variational within conformal classes.
Among the difficulties in applying these results are that it is unclear how to use Alexakis' approach to produce formally self-adjoint operators, or how to use Case, Lin, and Yuan's approach to identify the principal symbol of the corresponding operators.
These issues are more readily addressed by direct adaptation of the original construction of the GJMS operators.

The first systematic adaptation of the GJMS construction to nonlinear operators was carried out by Case, Lin, and Yuan~\cite{CaseLinYuan2022or}, who produced a family of conformally invariant bidifferential operators;
i.e.\ operators acting on pairs of functions that restrict to linear operators when one input is held fixed.
These operators generalize conformally invariant operators on the round sphere constructed by Ovsienko and Redou~\cite{OvsienkoRedou2003} for generic weights and by Clerc~\cites{Clerc2016,Clerc2017} for all weights, and are called \emph{curved Ovsienko--Redou operators}.
Like their counterparts on the sphere, the curved Ovsienko--Redou operators are constructed as polynomials in the ambient Laplacian.
Case, Lin, and Yuan used their construction to produce conformally invariant differential operators of homogeneity $-2k$ and order at most $2k-4$;
these can be added to the GJMS operators to produce new conformally invariant operators with leading-order term $\Delta^k$, strongly exhibiting the nonuniqueness of the GJMS operators.

In this paper we give the second systematic adaptation of the GJMS construction to nonlinear operators.
Our first result is the existence of families of conformally covariant tridifferential operators:

\begin{theorem}
    \label{tridifferential}
    Fix integers $k \geq 0$ and $n \geq 3$;
    if $n$ is even, then assume additionally that $n \geq 2k$.
    Let $w_1,w_2,w_3 \in \bR$.
    There is a $(k+1)$-dimensional family of functions $A \in \sF_k^5$ such that
    \begin{equation}
        \label{eqn:tridifferential-operator}
        \cD(\cu \otimes \cv \otimes \cw) := \sum_{\alpha \in \sI_k^5} \binom{k}{\alpha} A_{\alpha}\cDelta^{\alpha_1} \left( \cDelta^{\alpha_2} \bigl( (\cDelta^{\alpha_3}\cu)(\cDelta^{\alpha_4}\cv) \bigr) \cdot \cDelta^{\alpha_5} \cw \right)
    \end{equation}
    is tangential on $\cmE[w_1] \otimes \cmE[w_2] \otimes \cmE[w_3]$, where $A_\alpha := A(\alpha)$ and
    \begin{equation*}
     \binom{k}{\alpha} := \frac{k!}{\alpha_1!\alpha_2!\alpha_3!\alpha_4!\alpha_5!} .
    \end{equation*}
\end{theorem}

Here $\cDelta$ denotes the Laplacian on an ambient space $(\cmG,\cg)$ for a conformal $n$-manifold, $\cmE[w]$ denotes the space of smooth functions on $\cmG$ that are homogeneous of degree $w \in \bR$ with respect to dilations,
\begin{equation*}
 \sI_k^\ell := \left\{ (\alpha_1,\dotsc,\alpha_\ell) \in \bN_0^\ell \suchthat \alpha_1 + \dotsm + \alpha_\ell = k \right\}
\end{equation*}
denotes the set of $\ell$-tuples of nonnegative integers that sum to $k$,
\begin{equation*}
    \sF_k^\ell := \left\{ A \colon \sI_k^\ell \to \bR \right\}
\end{equation*}
is the vector space of scalar functions on $\sI_k^\ell$, and $\cu \cdot \cv$ denotes the pointwise product on $C^\infty(\cmG)$.
The assertion that $\cD$ is tangential means that $\cD$ descends to a conformally invariant tridifferential operator on the underlying conformal manifold.
See \cref{sec:bg} for further discussion.

The first part of the proof of \cref{tridifferential} is analogous to the construction of the GJMS operators by Graham et al.:
We use the aforementioned $\mathfrak{sl}_2$-triple to compute the commutators of $\cD$ with a particular defining function of the metric cone $\mG \subset \cmG$ and thereby derive a set of recurrence relations for the components $A_\alpha$ that are sufficient for $\cD$ to be tangential.
Our proof of the dimension estimate is new:
We realize solutions of these recurrence relations as elements of the kernel of the operator $d_1$ in a chain complex
\begin{equation}
    \label{eqn:chain-complex}
    0 \longrightarrow \sF_k^5 \overset{d_1}{\longrightarrow} \bigl(\sF_{k-1}^5\bigr)^3 \overset{d_2}{\longrightarrow} \bigl(\sF_{k-2}^5 \bigr)^3 \overset{d_3}{\longrightarrow} \sF_{k-3}^5 \longrightarrow 0 ,
\end{equation}
and then show that for generic weights $(w_1,w_2,w_3) \in \bR$, the cohomology of this complex is trivial except at $\sF_k^5$.
The conclusion then follows from a computation of the Euler characteristic and the upper semicontinuity of the dimension of the kernel of a continuous one-parameter family of linear maps.

Notably, \cref{tridifferential} always produces multidimensional families of operators, even on the sphere.
For comparison, the GJMS operators are unique up to scaling on the sphere~\cite{Branson1995}, while the Ovsienko--Redou operators are unique up to scaling for generic weights~\cite{OvsienkoRedou2003}.
The existence of such families was previously known in the special case $k=2$ and $w_1=w_2=w_3=-\frac{n-4}{2}$, where the $\sigma_2$-operator~\cite{Case2019fl}*{Remark~2.2} and the modification $u \mapsto uP_4u^2$ of the Paneitz operator give independent examples, but the dimension estimate is otherwise new.

Our proof of \cref{tridifferential} computes, for generic weights, the dimension of the space of operators on the sphere which can be written ambiently in the form of Equation~\eqref{eqn:tridifferential-operator}.
Inspired by Clerc's work on the Ovsienko--Redou operators~\cites{Clerc2016,Clerc2017}, we expect that the dimension is weight-dependent in general, but computable.
Finally, while one could avoid the chain complex~\eqref{eqn:chain-complex} and instead directly solve the recurrence relations, we believe that our approach is more robust.
To illustrate this point, in \cref{sec:tangential} we also construct new conformally invariant differential and bidifferential operators.

For applications to variational curvature prescription problems, one would like to know which of the operators constructed by \cref{tridifferential} are formally self-adjoint.
Two necessary conditions for formal self-adjointness are that $\cD$ is symmetric and that $w_1=w_2=w_3=-\frac{n-2k}{4}$, provided the operator is nonzero.
Remarkably, the symmetrizations of \emph{all} of the operators in Equation~\eqref{eqn:tridifferential-operator} are formally self-adjoint in the ambient space if this weight assumption is made.
Indeed:

\begin{theorem}
    \label{fsa}
    Fix integers $n > 2k \geq 2$.
    Let $A \in \sF_k^5$ be such that Equation~\eqref{eqn:tridifferential-operator} defines a tangential operator $\cD$ on $\cmE\bigl[-\frac{n-2k}{4}\bigr]^{\otimes3}$.
    Then
    \begin{equation*}
     \cD'(\cu_1 \otimes \cu_2 \otimes \cu_3) := \cD\bigl( \cu_1 \otimes \cu_2 \otimes \cu_3 \bigr) + \cD\bigl( \cu_2 \otimes \cu_3 \otimes \cu_1 \bigr) + \cD\bigl( \cu_3 \otimes \cu_1 \otimes \cu_2 \bigr) 
    \end{equation*}
    is formally self-adjoint and tangential on $\cmE\bigl[-\frac{n-2k}{4}\bigr]^{\otimes 3}$.
\end{theorem}

More precisely, the assumption of \cref{fsa} is that $A$ solves the recurrence relations~\eqref{eqn:recurrence}.
The assumption $n>2k$ can be relaxed to $n=2k$ after suitably picking $A_\alpha$ when one of $\alpha_1,\alpha_3,\alpha_4,\alpha_5=k$;
cf.\ \cref{rk:why-strict-inequality}.

\Cref{fsa} \emph{only} states that $\cD$ is formally self-adjoint, not that the induced operator $\iota^\ast\cD$ on the underlying conformal $n$-manifold is formally self-adjoint.
We do expect that $\iota^\ast\cD$ is formally self-adjoint.
The proof of \cref{fsa} is by an inductive argument that shows that the recursive relations defining $A_\alpha$ imply certain symmetries under the above weight assumptions.

There is a linear relation among the summands in the right-hand side of Equation~\eqref{eqn:tridifferential-operator} and the operators obtained by permuting $\cu_1,\cu_2,\cu_3$;
see \cref{rk:dimension}.
Based on this observation, we expect that the space of formally self-adjoint, conformally covariant, tridifferential operators of order $2k<n$ on the standard conformal $n$-sphere is $k$-dimensional.

For comparison, both the ambient operator $\cDelta^k$, which induces the GJMS operator of order $2k$, and the ambient operator
\begin{equation}
 \label{eqn:ovsienko-redou}
 \begin{split}
  \cD_{2k}(\cu \otimes \cv) & := \sum_{\alpha \in \sI_k^3} \binom{k}{\alpha}A_{\alpha}\cDelta^{\alpha_1}\bigl( (\cDelta^{\alpha_2}\cu) (\cDelta^{\alpha_3}\cv) \bigr) , \\
  A_{\alpha} & := \frac{\Gamma\bigl(\frac{n+4k}{6}-\alpha_1\bigr) \Gamma\bigl(\frac{n+4k}{6}-\alpha_2\bigr) \Gamma\bigl(\frac{n+4k}{6}-\alpha_3\bigr)}{\Gamma\bigl(\frac{n-2k}{6}\bigr) \Gamma\bigl(\frac{n+4k}{6}\bigr)^2} ,
 \end{split}
\end{equation}
which induces the curved Ovsienko--Redou operator of order $2k$, are formally self-adjoint in the ambient space.
It is known that the induced operators are formally self-adjoint, but the known proofs are quite involved:

The first proof of the formal self-adjointness of the GJMS operators used the identification of $P_{2k}$ in terms of the scattering operator for a Poincar\'e manifold and Green's theorem~\cite{GrahamZworski2003}.
While formal arguments allow one to bypass some of the analytic ingredients in scattering theory~\cite{FeffermanGraham2002}, the relation between $P_{2k}$ and the scattering operator is a manifestation of the realization of $P_{2k}$ as the obstruction to the existence of a smooth harmonic homogeneous extension~\cite{GJMS1992}.
This relation is not available for the curved Ovsienko--Redou operators or their related differential operators.
The formal self-adjointness of $P_{2k}$ can also be proved using Juhl's recursive formula in terms of formally self-adjoint second-order operators~\cites{Juhl2009,FeffermanGraham2013} or by renormalizing the energy of an operator of order $2k$ in the Poincar\'e space~\cite{CaseYan2024}, and both of these arguments generalize to the curved Ovsienko--Redou operators and their linear analogues~\cites{CaseYan2024,ChernYan2024}.
Both approaches involve complicated computations with many favorable cancellations.
It seems difficult to adapt these computations to the operators of \cref{fsa}, raising the question of whether there is a relatively simple explanation for these cancellations.
We leave this as an interesting open problem.

Our differential and bidifferential analogues of \cref{tridifferential} are also formally self-adjoint in the ambient space under the appropriate weight assumptions;
see \cref{sec:fsa} for details.

This article is organized as follows:

In \cref{sec:bg} we recall the key ingredients in the construction of conformally invariant polydifferential operators via tangential operators in the Fefferman--Graham ambient space.

In \cref{sec:tangential} we prove \cref{tridifferential}.
We also prove its (bi)differential analogues.

In \cref{sec:fsa} we prove \cref{fsa} and its (bi)differential analogues.

\section{Weyl operators}
\label{sec:bg}

In this section we collect required background on polydifferential operators and tangential operators in the ambient space.
Our discussion follows that of Case and Yan~\cite{CaseYan2024};
see also~\cites{CaseLinYuan2018b,CaseLinYuan2022or} for additional discussion of polydifferential operators and~\cites{FeffermanGraham2012,FeffermanHirachi2003,GJMS1992} for additional discussion of the ambient space and tangential operators thereon. 

A conformal manifold is a pair $(M^n,\kc)$ of a smooth $n$-manifold $M$ and a conformal class $\kc$;
i.e.\ $\kc$ is an equivalence class of (pseudo-)Riemannian metrics on $M$ for the equivalence relation $\hg \sim g$ if and only if there is an $\Upsilon \in C^\infty(M)$ such that $\hg = e^{2\Upsilon}g$.
We assume throughout that $n \geq 3$.
This data is equivalent to a choice of ray subbundle
\begin{equation*}
    \mG := \left\{ (x,g_x) \suchthat g \in \kc \right\} \subset S^2T^\ast M .
\end{equation*}
Define the projection $\pi \colon \mG \to M$ and the dilations $\delta_s \colon \mG \to \mG$, $s > 0$, by
\begin{align*}
    \pi(x,g_x) & := x , \\
    \delta_s(x,g_x) & := (x,s^2g_x) ,
\end{align*}
respectively.
Denote by $\bg$ the tautological section of $S^2T^\ast\mG$ defined by
\begin{equation*}
    \bg(X,Y) := g_x(\pi_\ast X , \pi_\ast Y)
\end{equation*}
for $X,Y \in T_{(x,g_x)}\mG$.

The space of \defn{conformal densities} of weight $w \in \bR$ on $(M^n,\kc)$ is
\begin{equation*}
    \mE[w] := \left\{ u \in C^\infty(\mG) \suchthat \text{$\delta_s^\ast u = s^wu$ for all $s>0$} \right\} .
\end{equation*}
A metric $g \in \kc$ determines a section of $\pi \colon \mG \to M$ by $g(x) := (x,g_x)$.
Pullback by $g$ induces a trivialization of $\mE[w]$;
i.e.\ the restriction $g^\ast \colon \mE[w] \to C^\infty(M)$ is bijective.
Note that if $u \in \mE[w]$ and $\hg = e^{2\Upsilon}g \in \kc$, then
\begin{equation}
    \label{eqn:pullback-on-densities}
    \hg^\ast u = e^{w\Upsilon}g^\ast u .
\end{equation}

An \defn{ambient space} for $(M^n,\kc)$ is a pseudo-Riemannian manifold $(\cmG^{n+2},\cg)$ together with a proper embedding $\iota \colon \mG \to \cmG$ and dilations $\cdelta_s \colon \cmG \to \cmG$, $s>0$, such that
\begin{enumerate}
    \item $\iota^\ast\cg = \bg$,
    \item $\cdelta_s^\ast\cg = s^2\cg$ for all $s>0$,
    \item $\iota \circ \delta_s = \cdelta_s \circ \iota$ for all $s > 0$, and
    \item $\Ric(\cg) = O^+(\rho^{(n-2)/2})$ if $n \geq 4$ is even, and $\Ric(\cg) = O(\rho^\infty)$ otherwise.
\end{enumerate}
Here $\Ric$ is the Ricci tensor, $\rho$ is a defining function for $\iota(\mG)$, and $O^+(\rho^m)$ is the set of all sections $T$ of $S^2T^\ast\cmG$ such that $T = O(\rho^m)$ and for every $z = (x,g_x) \in \mG$ there is a $t \in S^2T_x^\ast M$ such that $\iota_z^\ast(\rho^{-m}T) = \pi_z^\ast t$ and $\tr_{g_x}t = 0$.
Note that if $(\cmG,\cg)$ is an ambient space for $(M,\kc)$ and if $\cmU \subseteq \cmG$ is a dilation-invariant neighborhood of $\iota(\mG)$, then $(\cmU,\cg\rv_{\cmU})$ is also an ambient space for $(M,\kc)$.
Two ambient spaces $(\cmG_i,\cg_i)$, $i \in \{ 1,2 \}$, for $(M,\kc)$ are \defn{ambient equivalent} if, after shrinking $\cmG_1$ and $\cmG_2$ if necessary, there is a $\cdelta_s$-equivariant diffeomorphism $\Phi \colon \cmG_1 \to \cmG_2$ such that
\begin{enumerate}
    \item $\Phi \circ \iota = \iota$, and
    \item $\Phi^\ast\cg_2 - \cg_1 = O^+(\rho^{n/2})$ if $n \geq 4$ is even, and $\Phi^\ast\cg_2 - \cg_1 = O(\rho^\infty)$ otherwise.
\end{enumerate}
A fundamental result of Fefferman and Graham~\cite{FeffermanGraham2012}*{Theorem~2.3} states that every conformal manifold admits an ambient space and, moreover, this ambient space is unique up to ambient equivalence.

Given $w \in \bR$, denote by
\begin{equation*}
    \cmE[w] := \left \{ \cu \in C^\infty(\cmG) \suchthat \text{$\cdelta_s^\ast \cu = s^w \cu$ for all $s>0$} \right\}
\end{equation*}
the space of functions in the ambient space which have homogeneity $w$ with respect to dilations.
Note that $\iota^\ast \colon \cmE[w] \to \mE[w]$ is surjective and is natural with respect to ambient equivalence.
An \defn{extension} of $u \in \mE[w]$ is a choice of function $\cu \in \cmE[w]$ such that $u = \iota^\ast\cu$.

A \defn{polydifferential operator} of \defn{rank} $r \in \bN$ on $n$-manifolds is an assignment $D$ to each Riemannian manifold $(M^n,g)$ of a linear operator
\begin{equation*}
    D^g \colon C^\infty(M)^{\otimes(r-1)} \to C^\infty(M)
\end{equation*}
such that
\begin{equation*}
    D^g(u_1 \otimes \dotsm \otimes u_{r-1})
\end{equation*}
is given by a universal linear combination of complete contractions of covariant derivatives of the Riemann curvature tensor and of the inputs $u_1,\dotsc,u_{r-1}$, where $\Rm$ is regarded as a section of $\otimes^4T^\ast M$ and contractions are performed using $g^{-1}$.
A \defn{bidifferential} (resp.\ \defn{tridifferential}) operator is a polydifferential operator of rank $3$ (resp.\ rank $4$).
Note that if $D$ is a polydifferential operator of rank $r$ and if $\Phi \colon (M_1,g_1) \to (M_2,g_2)$ is an isometry, then
\begin{equation*}
    \Phi^\ast\bigl( D^{g_2}(u_1 \otimes \dotsm \otimes u_{r-1}) \bigr) = D^{g_1}\bigl( \Phi^\ast u_1 \otimes \dotsm \otimes \Phi^\ast u_{r-1} \bigr) 
\end{equation*}
for all $u_1,\dotsc,u_{r-1} \in C^\infty(M_2)$.
The \defn{Dirichlet form} $\kD$ associated to $D$ is the assignment to each Riemannian manifold $(M^n,g)$ of the linear functional
\begin{equation*}
    \kD \colon C_0^\infty(M)^{\otimes r} \to \bR
\end{equation*}
defined by
\begin{equation*}
    \kD^g( u_1 \otimes \dotsm \otimes u_r) := \int_M u_1 D^g( u_2 \otimes \dotsm \otimes u_r ) \dvol_g ,
\end{equation*}
where $C_0^\infty(M)$ denotes the space of compactly-supported functions on $M$.
We say that $D$ is \defn{formally self-adjoint} if $\kD$ is symmetric.
Note that if $D$ is formally self-adjoint, then it is symmetric.

A polydifferential operator $D$ has \defn{homogeneity} $h \in \bR$ if $D^{c^2g} = c^hD^g$ for all constants $c>0$ and all Riemannian manifolds $(M^n,g)$.
In this case we say that $D$ is \defn{homogeneous}.
Since $\nabla^k\Rm$, $\nabla^ku$, and $g^{-1}$ have homogeneity $2$, $0$, and $-2$, respectively, the homogeneity of a polydifferential operator, when defined, is always a nonpositive even integer.

An \defn{ambient polydifferential operator} on $n$-manifolds is a homogeneous polydifferential operator on $(n+2)$-manifolds.
We sometimes write
\begin{equation*}
    \cD \colon \cmE[w_1] \otimes \dotsm \otimes \cmE[w_{r-1}] \to \cmE[w-2k] ,
\end{equation*}
$w := \sum w_i$, as shorthand for an ambient polydifferential operator of rank $r$ and homogeneity $-2k$ on $n$-manifolds.
This notation will be used to specify situations in which $\cD$ is regarded as an assignment to each ambient space $(\cmG^{n+2},\cg)$ of a multilinear differential operator on homogeneous functions.

An ambient polydifferential operator
\begin{equation*}
    \cD \colon \cmE[w_1] \otimes \dotsm \otimes \cmE[w_{r-1}] \to \cmE[w - 2k]
\end{equation*}
on $n$-manifolds is \defn{tangential} if
\begin{equation}
    \label{eqn:restrict-ambient-polydifferential}
    \cu_1 \otimes \dotsm \otimes \cu_{r-1} \mapsto \iota^\ast \bigl( \cD^{\cg}( \cu_1 \otimes \dotsm \otimes \cu_{r-1}) \bigr)
\end{equation}
depends only on $\iota^\ast\cu_1,\dotsc,\iota^\ast\cu_{r-1}$ and is invariant with respect to ambient equivalence.
It follows that the operator
\begin{equation*}
    D \colon \mE[w_1] \otimes \dotsm \otimes \mE[w_{r-1}] \to \mE[ w - 2k]
\end{equation*}
defined by
\begin{equation*}
 D(u_1 \otimes \dotsm \otimes u_{r-1}) := \iota^\ast \bigl( \cD^{\cg}( \cu_1 \otimes \dotsm \otimes \cu_{r-1}) \bigr)
\end{equation*}
for some choice of extensions $\cu_j \in \cmE[w_j]$  of $u_j$ is well-defined.
We deduce from Equation~\eqref{eqn:pullback-on-densities} that the operator
\begin{equation*}
    D^g \colon C^\infty(M)^{\otimes(r-1)} \to C^\infty(M)
\end{equation*}
defined by
\begin{equation*}
    D^g(u_1 \otimes \dotsm \otimes u_{r-1}) := g^\ast\Bigl( D\bigl((g^\ast)^{-1}u_1 \otimes \dotsm \otimes (g^\ast)^{-1}u_{r-1}\bigr) \Bigr)
\end{equation*}
is conformally invariant:
if $\hg = e^{2\Upsilon}g$, then
\begin{equation*}
    D^{\hg}(u_1 \otimes \dotsm \otimes u_{r-1}) = e^{(w-2k)\Upsilon}D^g\bigl( e^{-w_1\Upsilon}u_1 \otimes \dotsm \otimes e^{-w_{r-1}\Upsilon}u_{r-1} \bigr) .
\end{equation*}
Note that if $D$ is formally self-adjoint and nonzero, then $w_1=\dotsm=w_{r-1}$ and $w_1+w-2k=-n$.
Hence $w_j = -\frac{n-2k}{r}$ for each $j \in \{ 1 , \dotsc , r-1 \}$.

Let $\cD$ be an ambient polydifferential operator of rank $r$ and homogeneity $-2k$ on $n$-manifolds.
In order to check that
\begin{equation*}
    \cD \colon \cmE[w_1] \otimes \dotsm \otimes \cmE[w_{r-1}] \to \cmE[ w - 2k ]
\end{equation*}
is tangential, it suffices to show that the commutators
\begin{multline*}
    [ \cD , \cQ ]_j \colon \cmE[w_1] \otimes \dotsm \otimes \cmE[w_{j-1}] \otimes \cmE[w_j-2] \otimes \cmE[w_{j+1}] \otimes \dotsm \otimes \cmE[w_{r-1}] \\ \to \cmE[ w - 2k]
\end{multline*}
vanish for each $j \in \{ 1, \dotsc, r-1 \}$, where $\cQ := \cg(\cX,\cX)$ for $\cX$ the infinitesimal generator of dilations and
\begin{multline*}
    [ \cD , \cQ]_j( \cu_1 \otimes \dotsm \otimes \cu_{r-1}) := \cD\bigl( \cu_1 \otimes \dotsm \otimes \cu_{j-1} \otimes \cQ\cu_j \otimes \cu_{j+1} \otimes \dotsm \otimes \cu_{r-1} \bigr) \\ - \cQ\cD\bigl( \cu_1 \otimes \dotsm \otimes \cu_{r-1}\bigr) :
\end{multline*}

\begin{proposition}
    \label{what-to-check-for-tangential}
    Fix integers $k \geq 0$ and $n \geq 3$.
    Let $\cD$ be an ambient polydifferential operator of rank $r$ and homogeneity $-2k$ on $n$-manifolds;
    if $n$ is even, then assume additionally that $n \geq 2k$.
    Suppose that $w_1,\dotsc,w_{r-1} \in \bR$ are such that if $j \in \{ 1, \dotsc, r-1 \}$, then
    \begin{equation*}
        [ \cD , \cQ ]_j = 0
    \end{equation*}
    on $\cmE[w_1] \otimes \dotsm \otimes \cmE[w_{j-1}] \otimes \cmE[w_j-2] \otimes \cmE[w_{j+1}] \otimes \dotsm \otimes \cmE[w_{r-1}]$.
    Then
    \begin{equation*}
        \cD \colon \cmE[w_1] \otimes \dotsm \otimes \cmE[w_{r-1}] \to \cmE[w-2k]
    \end{equation*}
    is tangential.
\end{proposition}

\begin{proof}
    Case and Yan showed that~\eqref{eqn:restrict-ambient-polydifferential} is invariant with respect to ambient equivalence~\cite{CaseYan2024}*{Corollary~2.3}.

    Let $\cu_j \in \cmE[w_j]$, $j \in \{ 1, \dotsc, r-1 \}$.
    Suppose there is an $i \in \{ 1, \dotsc, r-1 \}$ such that $\iota^\ast\cu_i = 0$.
    Since $\cQ$ is a defining function for $\iota(\mG)$~\cite{GJMS1992}*{p.\ 560}, there is a $\cv \in \cmE[w_i-2]$ such that $\cu_i=\cQ\cv$.
    The vanishing of $[\cD,\cQ]_i$ implies that
    \begin{equation*}
        \cD(\cu_1 \otimes \dotsm \otimes \cu_{r-1}) = \cQ\cD(\cu_1 \otimes \dotsm \otimes \cu_{i-1} \otimes \cv \otimes \cu_{i+1} \otimes \dotsm \otimes \cu_{r-1}) .
    \end{equation*}
    In particular,
    \begin{equation*}
        \iota^\ast\bigl( \cD( \cu_1 \otimes \dotsm \otimes \cu_{r-1}) \bigr) = 0.
    \end{equation*}
    We conclude from linearity that~\eqref{eqn:restrict-ambient-polydifferential} depends only on $\iota^\ast\cu_1,\dotsc,\iota^\ast\cu_{r-1}$.
\end{proof}

For ambient polydifferential operators built from the ambient Laplacian and multiplication by scalar Riemannian invariants, one can use the fact that if $(\cmG^{n+2},\cg)$ is a straight ambient space, then for each $k \in \bN$ it holds that\footnote{
    We use the convention that $\Delta = -\nabla^a\nabla_a$, so that $\Delta \geq 0$ in Riemannian signature.
    This is the oppositive convention from~\cite{GJMS1992}, and is responsible for the sign in Equation~\eqref{eqn:fundamental-commutator}.
}
\begin{equation}
    \label{eqn:fundamental-commutator}
    [ \cDelta^k , \cQ ] = -2k\cDelta^{k-1}(2\cX + n + 4 - 2k)
\end{equation}
as operators on $C^\infty(\cmG)$~\cite{GJMS1992}*{Equation~(1.8)}.
Since one can always choose the ambient space to be straight~\cite{FeffermanGraham2012}*{Proposition~2.6}, and since $\cX\cu = w\cu$ if and only if $\cu \in \cmE[w]$, \cref{what-to-check-for-tangential} and Equation~\eqref{eqn:fundamental-commutator} give an efficient way to check tangentiality.
This is the strategy used in the GJMS~\cite{GJMS1992} and Case--Lin--Yuan~\cite{CaseLinYuan2022or} constructions, and is also the strategy we use in \cref{sec:tangential}.

\section{Conformally covariant tridifferential operators}
\label{sec:tangential}

In this section we prove \cref{tridifferential}. We do so by proving that the hypotheses of \cref{what-to-check-for-tangential} are satisfied if the coefficients of $\cD$ satisfy three recurrence relations. Solutions of the recurrence relations correspond to elements of the kernel of a linear map $d_{1,\boldsymbol{w}} \colon \sF_k^5 \to (\sF_{k-1}^5)^3$.
The theorem follows by computing $\dim \ker d_{1,\boldsymbol{w}}$ for generic weights $\boldsymbol{w}$.

In the remainder of this paper, we denote by $\vec{e}_j \in \bN_0^\ell$ the $\ell$-tuple
\begin{equation*}
    \vec{e}_j := (0,\dotsc,0,1,0,\dotsc,0)
\end{equation*}
with a $1$ in the $j$-th component and a $0$ everywhere else.
These are helpful in relating elements of $\sI_{k-1}^\ell$ and $\sI_k^\ell$.

\begin{proof}[Proof of \cref{tridifferential}]
    Define $\cD$ by Equation~\eqref{eqn:tridifferential-operator}.
    By \cref{what-to-check-for-tangential}, it suffices to check that the commutators $[\cD,\cQ]_j$, $j \in \{ 1, 2, 3 \}$, all vanish.

    The product rule for commutators gives
    \begin{equation*}
        \begin{split}
            [\cD, \cQ]_1(\cu \otimes \cv \otimes \cw) & = \sum_{\alpha\in\sI_k^5} \binom{k}{\alpha} A_{\alpha} \biggl[ \cDelta^{\alpha_1}\left( \cDelta^{\alpha_2} \bigl( ( [\cDelta^{\alpha_3}, \cQ]\cu ) (\cDelta^{\alpha_4}\cv) \bigr) \cdot \cDelta^{\alpha_5}\cw \right) \\
            & \qquad + \cDelta^{\alpha_1}\left( \bigl[\cDelta^{\alpha_2}, \cQ\bigr] \bigl( (\cDelta^{\alpha_3}\cu)(\cDelta^{\alpha_4}\cv) \bigr) \cdot \cDelta^{\alpha_5}\cw \right)\\
            & \qquad + \bigl[\cDelta^{\alpha_1}, \cQ\bigr] \left( \cDelta^{\alpha_2} \bigl( (\cDelta^{\alpha_3}\cu)(\cDelta^{\alpha_4}\cv) \bigr) \cdot \cDelta^{\alpha_5}\cw \right) \biggr] .
        \end{split}
    \end{equation*}
    Combining this with Equation~\eqref{eqn:fundamental-commutator} yields
    \begin{align*}
        \MoveEqLeft[1] [\cD, \cQ]_1(\cu \otimes \cv \otimes \cw) = -4\sum_{\alpha\in\sI_k^5} \binom{k}{\alpha}A_{\alpha} \times \biggl[ \\
        & \alpha_3\left(\frac{n}{2}+w_1-\alpha_3\right) \cDelta^{\alpha_1}\left( \cDelta^{\alpha_2} \bigl( ( \cDelta^{\alpha_3-1}\cu ) (\cDelta^{\alpha_4}\cv) \bigr) \cdot \cDelta^{\alpha_5}\cw \right) \\
        & + \alpha_2\left(\frac{n}{2}+w_1+w_2-2\alpha_3-2\alpha_4-\alpha_2\right) \cDelta^{\alpha_1}\left( \cDelta^{\alpha_2-1}\bigl( (\cDelta^{\alpha_3}\cu)(\cDelta^{\alpha_4}\cv) \bigr) \cdot \cDelta^{\alpha_5}\cw \right)\\
        & + \alpha_1\left(\frac{n}{2}+w-2k+\alpha_1\right) \cDelta^{\alpha_1-1}\left( \cDelta^{\alpha_2} \bigl( (\cDelta^{\alpha_3}\cu)(\cDelta^{\alpha_4}\cv) \bigr) \cdot \cDelta^{\alpha_5}\cw \right) \biggr]
    \end{align*}
    on $\cmE[w_1-2] \otimes \cmE[w_2] \otimes \cmE[w_3]$, where $w:=w_1+w_2+w_3$.
    Equivalently,
    \begin{equation*}
        [\cD,\cQ]_1(\cu \otimes \cv \otimes \cw) = -4k\sum_{\alpha \in \sI_{k-1}^5} \binom{k-1}{\alpha}B_\alpha^{(1)}\cDelta^{\alpha_1}\left( \cDelta^{\alpha_2} \bigl( (\cDelta^{\alpha_3}\cu) (\cDelta^{\alpha_4}\cv) \bigr) \cdot \cDelta^{\alpha_5}\cw \right) ,
    \end{equation*}
    where $B^{(1)} \in \sF_{k-1}^5$ is defined by
    \begin{equation}
        \label{eqn:tridifferential-B1}
        \begin{split}
            B_\alpha^{(1)} & := \left(\frac{n}{2}+w_1-\alpha_3-1\right)A_{\alpha+\vec{e}_3} + \left(\frac{n}{2}+w-2k+\alpha_1+1\right)A_{\alpha+\vec{e}_1} \\
            & \quad + \left(\frac{n}{2} + w_1 + w_2 - \alpha_2 - 2\alpha_3 - 2\alpha_4 - 1\right)A_{\alpha+\vec{e}_2} .
        \end{split}
    \end{equation}
    Similarly, we compute that
    \begin{align*}
        [\cD,\cQ]_2(\cu \otimes \cv \otimes \cw) & = -4k\sum_{\alpha \in \sI_{k-1}^5} \binom{k-1}{\alpha}B_\alpha^{(2)}\cDelta^{\alpha_1}\left( \cDelta^{\alpha_2} \bigl( (\cDelta^{\alpha_3}\cu) (\cDelta^{\alpha_4}\cv) \bigr) \cdot \cDelta^{\alpha_5}\cw \right) , \\
        [\cD,\cQ]_3(\cu \otimes \cv \otimes \cw) & = -4k\sum_{\alpha \in \sI_{k-1}^5} \binom{k-1}{\alpha}B_\alpha^{(3)}\cDelta^{\alpha_1}\left( \cDelta^{\alpha_2} \bigl( (\cDelta^{\alpha_3}\cu) (\cDelta^{\alpha_4}\cv) \bigr) \cdot \cDelta^{\alpha_5}\cw \right) ,
    \end{align*}
    where $B^{(2)},B^{(3)} \in \sF_{k-1}^5$ are defined by
    \begin{align}
        \label{eqn:tridifferential-B2} B_\alpha^{(2)} & := \left(\frac{n}{2} + w_2 - \alpha_4 - 1\right)A_{\alpha+\vec{e}_4} + \left(\frac{n}{2} + w - 2k + \alpha_1 + 1 \right)A_{\alpha + \vec{e}_1} \\
            \notag & \quad + \left(\frac{n}{2} + w_1 + w_2 - \alpha_2 - 2\alpha_3 - 2\alpha_4 - 1 \right)A_{\alpha + \vec{e}_2}, \\
        \label{eqn:tridifferential-B3} B_\alpha^{(3)} & := \left( \frac{n}{2} + w_3 - \alpha_5 - 1 \right)A_{\alpha + \vec{e}_5} + \left( \frac{n}{2} + w - 2k + \alpha_1 + 1 \right)A_{\alpha + \vec{e}_1} .
    \end{align}
    We deduce that $\cD$ is tangential if $B^{(1)},B^{(2)},B^{(3)}=0$.

    Given $\boldsymbol{w} = (w_1,w_2,w_3) \in \bR^3$, define $F_{j,\boldsymbol{w}},F_{2,\boldsymbol{w}}' \colon \sF_s^5 \to \sF_{s-1}^5$, $j \in \{ 1, \dotsc, 5 \}$, by
    \begin{align*}
        F_{1,\boldsymbol{w}}(A)_\alpha & := \left( \frac{n}{2} + w - 2k + \alpha_1 + 1 \right)A_{\alpha + \vec{e}_1} , \\
        F_{2,\boldsymbol{w}}(A)_\alpha & := \left( \frac{n}{2} + w_1 + w_2 - \alpha_2 - 2\alpha_3 - 2\alpha_4 -1 \right)A_{\alpha + \vec{e}_2} , \\
        F_{2,\boldsymbol{w}}'(A)_\alpha & := \left( \frac{n}{2} + w_1 + w_2 - \alpha_2 - 2\alpha_3 - 2\alpha_4 - 3 \right)A_{\alpha + \vec{e}_2} , \\
        F_{3,\boldsymbol{w}}(A)_\alpha & := \left( \frac{n}{2} + w_1 - \alpha_3 - 1 \right)A_{\alpha + \vec{e}_3} , \\
        F_{4,\boldsymbol{w}}(A)_\alpha & := \left( \frac{n}{2} + w_2 - \alpha_4 - 1 \right)A_{\alpha + \vec{e}_4} , \\
        F_{5,\boldsymbol{w}}(A)_\alpha & := \left( \frac{n}{2} + w_3 - \alpha_5 - 1 \right)A_{\alpha + \vec{e}_5} .
    \end{align*}
    It is straightforward to check that
    \begin{align*}
        [F_{i,\boldsymbol{w}} , F_{j,\boldsymbol{w}}] & = 0 , && \text{if $i,j \in \{ 1, 3, 4, 5 \}$} , \\
        [F_{i,\boldsymbol{w}} , F_{2,\boldsymbol{w}}] & = 0 , && \text{if $i \in \{ 1, 5 \}$} , \\
        F_{i,\boldsymbol{w}}F_{2,\boldsymbol{w}} & = F_{2,\boldsymbol{w}}'F_{i,\boldsymbol{w}} , && \text{if $i \in \{ 3, 4 \}$} .
    \end{align*}
    Therefore
    \begin{equation}
        \label{eqn:tridifferential-complex}
        0 \longrightarrow \sF_k^5 \overset{d_{1,\boldsymbol{w}}}{\longrightarrow} \bigl(\sF_{k-1}^5\bigr)^3 \overset{d_{2,\boldsymbol{w}}}{\longrightarrow} \bigl(\sF_{k-2}^5\bigr)^3 \overset{d_{3,\boldsymbol{w}}}{\longrightarrow} \sF_{k-3}^5 \longrightarrow 0
    \end{equation}
    is a complex for all $\boldsymbol{w} = (w_1,w_2,w_3) \in \bR^3$, where
    \begin{align*}
        d_{1,\boldsymbol{w}} & := \begin{pmatrix} F_{1,\boldsymbol{w}} + F_{2,\boldsymbol{w}} + F_{3,\boldsymbol{w}} \\ F_{1,\boldsymbol{w}} + F_{2,\boldsymbol{w}} + F_{4,\boldsymbol{w}} \\ F_{1,\boldsymbol{w}} + F_{5,\boldsymbol{w}} \end{pmatrix} , \\
        d_{2,\boldsymbol{w}} & := \begin{pmatrix} -F_{1,\boldsymbol{w}}-F_{5,\boldsymbol{w}} & F_{1,\boldsymbol{w}}+F_{5,\boldsymbol{w}} & F_{3,\boldsymbol{w}}-F_{4,\boldsymbol{w}} \\ F_{1,\boldsymbol{w}} + F_{5,\boldsymbol{w}} & 0 & -F_{1,\boldsymbol{w}}-F_{2,\boldsymbol{w}}-F_{3,\boldsymbol{w}} \\ -F_{1,\boldsymbol{w}} - F_{2,\boldsymbol{w}}'-F_{4,\boldsymbol{w}} & F_{1,\boldsymbol{w}} + F_{2,\boldsymbol{w}}' + F_{3,\boldsymbol{w}} & 0 \end{pmatrix} , \\
        d_{3,\boldsymbol{w}} & := \begin{pmatrix} F_{1,\boldsymbol{w}} + F_{2,\boldsymbol{w}}' + F_{3,\boldsymbol{w}} & F_{3,\boldsymbol{w}} - F_{4,\boldsymbol{w}} & -F_{1,\boldsymbol{w}} - F_{5,\boldsymbol{w}} \end{pmatrix} ,
    \end{align*}
    and elements of $(\sF_{k-1}^5)^3$ and $(\sF_{k-2}^5)^3$ are represented as column vectors.
    On the one hand, the elementary formula
    \begin{equation*}
        \lv \sI_k^\ell \rv = \binom{k+\ell-1}{\ell-1}
    \end{equation*}
    implies that the Euler characteristic of Chain Complex~\eqref{eqn:tridifferential-complex} is
    \begin{equation}
        \label{eqn:euler-characteristic1}
        \chi = \lv \sF_k^5 \rv - 3\lv \sF_{k-1}^5 \rv + 3\lv\sF_{k-2}^5 \rv - \lv\sF_{k-3}^5\rv = k + 1 .
    \end{equation}
    On the other hand, the previous paragraph implies that if $A \in \ker d_{1,\boldsymbol{w}}$, then Equation~\eqref{eqn:tridifferential-operator} is tangential.
    In the next paragraph, we prove that there is a dense set $W \subset \bR^3$ such that if $(w_1,w_2,w_3) \in W$, then
    \begin{equation}
        \label{eqn:defn-W}
        \ker d_{2,\boldsymbol{w}} = \im d_{1,\boldsymbol{w}} , \quad \ker d_{3,\boldsymbol{w}} = \im d_{2,\boldsymbol{w}} , \quad \sF_{k-3}^5 = \im d_{3,\boldsymbol{w}} .
    \end{equation}
    It follows that the Euler characteristic of Chain Complex~\eqref{eqn:tridifferential-complex} is
    \begin{equation}
        \label{eqn:euler-characteristic2}
        \chi = \dim \ker d_{1,\boldsymbol{w}}
    \end{equation}
    for all $\boldsymbol{w} \in W$.
    Since $\boldsymbol{w} \mapsto d_{1,\boldsymbol{w}}$ is continuous, $\boldsymbol{w} \mapsto \dim \ker d_{1,\boldsymbol{w}}$ is upper semicontinuous.
    Combining this with Equations~\eqref{eqn:euler-characteristic1} and~\eqref{eqn:euler-characteristic2} yields
    \begin{equation*}
        \dim \ker d_{1,\boldsymbol{w}} \geq k+1
    \end{equation*}
    for all $\boldsymbol{w} \in \bR^3$.

    Consider the set
    \begin{equation*}
        W := \left\{ (w_1,w_2,w_3) \in \bR^3 \suchthat 2w_1 , 2w_2, 2w_3 \not\in \bZ \right\} .
    \end{equation*}
    Clearly $W \subset \bR^3$ is dense;
    we will show that Conditions~\eqref{eqn:defn-W} hold for all $\boldsymbol{w} \in W$.
    Let $\boldsymbol{w} \in W$.
    Then $F_{j,\boldsymbol{w}},F_{2,\boldsymbol{w}}' \colon \sF_s^5 \to \sF_{s-1}^5$ is invertible for each $j \in \{ 1, \dotsc, 5 \}$ and each $s \in \bN$.
    Let $G_{j,\boldsymbol{w}},G_{2,\boldsymbol{w}}' \colon \sF_{s-1}^5 \to \sF_s^5$ denote the right inverses for which $(G_{j,\boldsymbol{w}}'A)_\alpha = 0$ and $(G_{2,\boldsymbol{w}}'A)_\alpha=0$ whenever $\alpha_j=0$ and $\alpha_2=0$, respectively.
    It follows that
    \begin{align*}
        [G_{i,\boldsymbol{w}} , G_{j,\boldsymbol{w}}] & = 0 , && \text{if $i,j \in \{ 1, 3, 4, 5 \}$} , \\
        [G_{i,\boldsymbol{w}} , G_{2,\boldsymbol{w}}] & = 0 , && \text{if $i \in \{ 1, 5 \}$} , \\
        G_{i,\boldsymbol{w}}G_{2,\boldsymbol{w}}' & = G_{2,\boldsymbol{w}}G_{i,\boldsymbol{w}} , && \text{if $i \in \{ 3, 4 \}$} .
    \end{align*}
    We also have $F_iG_2 = G_2'F_i$ and $G_iF_2'=F_2G_i$ for $i \in \{ 3, 4 \}$, and all other pairs $F_i,G_j$ with $i\not=j$ commute.
    Therefore the operators appearing in the components of the matrices defining $d_{j,\boldsymbol{w}}$ are all invertible;
    e.g.
    \begin{align*}
        (F_{1,\boldsymbol{w}} + F_{5,\boldsymbol{w}})^{-1} & = \sum_{j=0}^\infty (-1)^j G_{5,\boldsymbol{w}}(F_{1,\boldsymbol{w}}G_{5,\boldsymbol{w}})^{j} , \\
        (F_{1,\boldsymbol{w}} + F_{2,\boldsymbol{w}} + F_{3,\boldsymbol{w}})^{-1} & = \sum_{j=0}^\infty (-1)^jG_{1,\boldsymbol{w}}\bigl((F_{2,\boldsymbol{w}}+F_{3,\boldsymbol{w}})G_{1,\boldsymbol{w}}\bigr)^j .
    \end{align*}
    It follows immediately that $d_{3,\boldsymbol{w}}$ is surjective.
    We readily check that
    \begin{align*}
        (F_1+F_5)(F_3-F_4)^{-1} & = (F_3-F_4)^{-1}(F_1+F_5) , \\
        (F_1+F_2+F_3)(F_3-F_4)^{-1} & = (F_3-F_4)^{-1}(F_1+F_2'+F_3) , \\
        (F_3-F_4)(F_1+F_2+F_3)^{-1} & = (F_1+F_2'+F_3)^{-1}(F_3-F_4) .
    \end{align*}
    Direct computation implies that if $(C_1,C_2,C_3) \in \ker d_{3,\boldsymbol{w}}$, then
    \begin{equation*}
        d_{2,\boldsymbol{w}}\begin{pmatrix} (F_3-F_4)^{-1}C_3 \\ (F_3-F_4)^{-1}C_3 \\ (F_3-F_4)^{-1}C_1 - (F_1+F_2+F_3)^{-1}T_1C_2 \end{pmatrix} = \begin{pmatrix} C_1 \\ C_2 \\ C_3 \end{pmatrix} ,
    \end{equation*}
    where $T_1 := I - (F_3-F_4)^{-1}(F_3-F_4)$.
    Direct computation also implies that if $(B_1,B_2,B_3) \in \ker d_{2,\boldsymbol{w}}$, then
    \begin{equation*}
        A := (F_1+F_5)^{-1}B_3 + (F_3-F_4)^{-1}T_2(B_1-B_2) + (F_1+F_2+F_3)^{-1}T_1T_2B_1
    \end{equation*}
    is such that $d_{1,\boldsymbol{w}}A = (B_1,B_2,B_3)$, where $T_2 := I - (F_1+F_5)^{-1}(F_1+F_5)$.
    Thus $W$ has the desired properties.
\end{proof}

We conclude this section by adapting the above proof to construct some new conformally invariant differential and bidifferential operators.

We first construct a family of differential operators that is more general than those of Case, Lin, and Yuan~\cite{CaseLinYuan2022or}:

\begin{proposition}
    \label{linear-analogue}
    Fix integers $n \geq 3$ and $k,\ell_1,\ell_2 \geq 0$ such that $k \geq \ell := \ell_1+\ell_2$;
    if $n$ is even, then assume additionally that $n \geq 2k$.
    Fix $w \in \bR$ and let $A \in \sF_{k-\ell}^3$ be a solution of the recurrence relation
    \begin{multline}
        \label{eqn:linear-analogue-recurrence}
        0 = \left( \frac{n}{2} + w - \alpha_3 - 1 \right)A_{\alpha+\vec{e}_3} + \left( \frac{n}{2} + w - 2k + \alpha_1 + 1 \right)A_{\alpha+\vec{e}_1} \\
         + \left( \frac{n}{2} + w - 2\ell_2 - \alpha_2 - 2\alpha_3 - 1 \right)A_{\alpha+\vec{e}_2} .
    \end{multline}
    If $\cI_1,\cI_2$ are scalar Riemannian invariants of weight $-2\ell_1$ and $-2\ell_2$, respectively, on $(n+2)$-manifolds, then
    \begin{equation}
        \label{eqn:linear-analogue}
        \cD(\cu ) := \sum_{\alpha \in \sI_{k-\ell}^3} \binom{k-\ell}{\alpha}A_\alpha \cDelta^{\alpha_1}\bigl( \cI_1 \cDelta^{\alpha_2}( \cI_2\cDelta^{\alpha_3}\cu ) \bigr)
    \end{equation}
    is a tangential differential operator on $\cmE[w]$.
    Moreover, the space of solutions $A$ of the recurrence relation~\eqref{eqn:linear-analogue-recurrence} is at least $(k-\ell+1)$-dimensional.
\end{proposition}

\begin{proof}
    Define $\cD$ by Equation~\eqref{eqn:linear-analogue}.
    Direct computation yields
    \begin{align*}
        \bigl[ \cD , \cQ \bigr](\cu) & = -4(k-\ell)\sum_{\alpha \in \sI_{k-\ell-1}^3} \binom{k-\ell-1}{\alpha}B_\alpha \cDelta^{\alpha_1}\bigl( \cI_1 \cDelta^{\alpha_2}( \cI_2\cDelta^{\alpha_3}\cu ) \bigr) , \\
        B_\alpha & = \left( \frac{n}{2} + w - \alpha_3 - 1 \right)A_{\alpha+\vec{e}_3} + \left( \frac{n}{2} + w - 2k + \alpha_1 + 1 \right)A_{\alpha+\vec{e}_1} \\
         & \quad + \left( \frac{n}{2} + w - 2\ell_2 - 2\alpha_3 - \alpha_2 - 1 \right)A_{\alpha+\vec{e}_2}
    \end{align*}
    on $\cmE[w-2]$.
    It follows immediately from \cref{what-to-check-for-tangential} and our assumption on $A$ that $\cD$ is tangential on $\cmE[w]$.

    Given $w \in \bR$, define $F_{j,w} \colon \sF_s^3 \to \sF_{s-1}^3$, $j \in \{ 1, 2, 3 \}$ and $s \in \bN$, by
    \begin{align*}
        F_{1,w}(A)_\alpha & := \left(\frac{n}{2} + w - 2k + \alpha_1 + 1 \right)A_{\alpha + \vec{e}_1} , \\
        F_{2,w}(A)_{\alpha} & := \left( \frac{n}{2} + w - 2\ell_2 - \alpha_2 - 2\alpha_3 - 1 \right)A_{\alpha + \vec{e}_2} , \\
        F_{3,w}(A)_{\alpha} & := \left( \frac{n}{2}+w-\alpha_3-1 \right)A_{\alpha + \vec{e}_3} .
    \end{align*}
    Clearly
    \begin{equation*}
        0 \longrightarrow \sF_k^3 \overset{d_{1,w}}{\longrightarrow} \sF_{k-1}^3 \longrightarrow 0
    \end{equation*}
    is a chain complex, where $d_{1,w} := F_{1,w} + F_{2,w} + F_{3,w}$.
    Since each element of $\ker d_{1,w}$ gives a solution of the recurrence relation~\eqref{eqn:linear-analogue-recurrence}, an argument as in the proof of \cref{tridifferential} implies that
    \begin{equation*}
        \dim \ker d_{1,w} \geq \lv \sF_{k-\ell}^3 \rv - \lv \sF_{k-\ell-1}^3 \rv = k-\ell+1 . \qedhere
    \end{equation*}
\end{proof}

We now construct three different families of conformally invariant bidifferential operators, each of which exhibits the significant lack of uniqueness for the curved Ovsienko--Redou operators.

Our first bidifferential operators most closely resemble the ambient formula~\eqref{eqn:ovsienko-redou} for the curved Ovsienko--Redou operators:

\begin{proposition}
    \label{outer-ovsienko-redou}
    Fix integers $k \geq \ell \geq 0$ and $n \geq 3$;
    if $n$ is even, then assume additionally that $n \geq 2k$.
    Fix $w_1,w_2 \in \bR$ and let $A \in \sF_{k-\ell}^3$ be a solution of the recurrence relations
    \begin{equation}
        \label{eqn:outer-ovsienko-redou-recurrence}
        \begin{split}
            \left( \frac{n}{2} + w_1 - \alpha_2 - 1 \right)A_{\alpha + \vec{e}_2} & = -\left( \frac{n}{2} + w - 2k + \alpha_1 + 1 \right) A_{\alpha + \vec{e}_1} , \\
            \left( \frac{n}{2} + w_2 - \alpha_3 - 1 \right)A_{\alpha + \vec{e}_3} & = -\left( \frac{n}{2} + w - 2k + \alpha_1 + 1 \right) A_{\alpha+\vec{e}_1} ,
        \end{split}
    \end{equation}
    for $w := w_1 + w_2$.
    If $\cI$ is a scalar Riemannian invariant of weight $-2\ell$ on $(n+2)$-manifolds, then
    \begin{equation}
        \label{eqn:outer-ovsienko-redou}
        \cD(\cu \otimes \cv) := \sum_{\alpha\in \sI_{k-\ell}^3} \binom{k-\ell}{\alpha}A_{\alpha}\cDelta^{\alpha_1}\left( \cI \bigl( \cDelta^{\alpha_2}\cu \bigr) \bigl( \cDelta^{\alpha_3}\cv \bigr) \right)
    \end{equation}
    is a tangential ambient bidifferential operator on $\cmE[w_1] \otimes \cmE[w_2]$.
    Moreover, the space of solutions of the recurrence relations~\eqref{eqn:outer-ovsienko-redou-recurrence} is at least $1$-dimensional.
\end{proposition}

\begin{proof}
    Define $\cD$ by Equation~\eqref{eqn:outer-ovsienko-redou}.
    Direct computation yields
    \begin{equation*}
        \bigl[ \cD , \cQ \bigr]_j = -4(k-\ell)\sum_{\alpha\in\sI_{k-\ell-1}^3} \binom{k-\ell-1}{\alpha} B_\alpha^{(j)} \cDelta^{\alpha_1} \left( \cI \bigl( \cDelta^{\alpha_2}\cu \bigr) \bigl( \cDelta^{\alpha_3}\cu \bigr) \right)
    \end{equation*}
    on $\cmE[w_1-2] \otimes \cmE[w_2]$ when $j=1$, and on $\cmE[w_1] \otimes \cmE[w_2-2]$ when $j=2$, where
    \begin{align*}
        B_\alpha^{(1)} & = \left( \frac{n}{2} + w_1 - \alpha_2 - 1 \right)A_{\alpha+\vec{e}_2} + \left( \frac{n}{2} + w - 2k + \alpha_1 + 1 \right) A_{\alpha+\vec{e}_1} , \\
        B_\alpha^{(2)} & = \left( \frac{n}{2} + w_2 - \alpha_3 - 1 \right)A_{\alpha+\vec{e}_3} + \left( \frac{n}{2} + w - 2k + \alpha_1 + 1 \right) A_{\alpha+\vec{e}_1} .
    \end{align*}
    It follows immediately from \cref{what-to-check-for-tangential} and our assumptions on $A$ that $\cD$ is tangential on $\cmE[w_1] \otimes \cmE[w_2]$.

    Given $\boldsymbol{w} \in \bR^2$, define $F_{j,\boldsymbol{w}} \colon \sF_s^3 \to \sF_{s-1}^3$, $j \in \{ 1, 2, 3 \}$ and $s \in \bN$, by
    \begin{align*}
        F_{1,\boldsymbol{w}}(A)_\alpha & := \left(\frac{n}{2} + w - 2k + \alpha_1 + 1 \right)A_{\alpha + \vec{e}_1} , \\
        F_{2,\boldsymbol{w}}(A)_\alpha & := \left(\frac{n}{2} + w_1 - \alpha_2 - 1 \right)A_{\alpha + \vec{e}_2} , \\
        F_{3,\boldsymbol{w}}(A)_{\alpha} & := \left( \frac{n}{2} + w_2 - \alpha_3 - 1 \right)A_{\alpha + \vec{e}_3} .
    \end{align*}
    Direct computation implies that
    \begin{equation*}
        0 \longrightarrow \sF_{k-\ell}^3 \overset{d_{1,\boldsymbol{w}}}{\longrightarrow} \bigl(\sF_{k-\ell-1}^3\bigr)^2 \overset{d_{2,\boldsymbol{w}}}{\longrightarrow} \sF_{k-\ell-2}^3 \longrightarrow 0
    \end{equation*}
    is a chain complex, where
    \begin{align*}
        d_{1,\boldsymbol{w}} & := \begin{pmatrix} F_{1,\boldsymbol{w}} + F_{2,\boldsymbol{w}} & F_{1,\boldsymbol{w}} + F_{3,\boldsymbol{w}} \end{pmatrix}^T , \\
        d_{2,\boldsymbol{w}} & := \begin{pmatrix} F_{1,\boldsymbol{w}} + F_{3,\boldsymbol{w}} & -( F_{1,\boldsymbol{w}} + F_{2,\boldsymbol{w}}) \end{pmatrix} .
    \end{align*}
    Since each element of $\ker d_{1,\boldsymbol{w}}$ gives a solution of the recurrence relation~\eqref{eqn:outer-ovsienko-redou-recurrence}, an argument as in the proof of Theorem~\ref{tridifferential} implies that
    \begin{equation*}
        \ker \dim d_{1,\boldsymbol{w}} \geq \lv\sF_{k-\ell}^3\rv - 2\lv\sF_{k-\ell-1}^3\rv + \lv\sF_{k-\ell-2}^3\rv = 1 . \qedhere
    \end{equation*}
\end{proof}

Our second bidifferential operator modifies a curved Ovsienko--Redou operator by replacing one of the inner Laplacians by a sum $\sum \cDelta^{\alpha_1} \circ I \circ \cDelta^{\alpha_2}$:

\begin{proposition}
    \label{inner-ovsienko-redou}
    Fix integers $k \geq \ell \geq 0$ and $n \geq 3$;
    if $n$ is even, then assume additionally that $n \geq 2k$.
    Fix $w_1,w_2 \in \bR$ and let $A \in \sF_{k-\ell}^4$ be a solution of the recurrence relations
    \begin{equation}
        \label{eqn:inner-ovsienko-redou-recurrence}
        \begin{split}
            \MoveEqLeft[10] \left( \frac{n}{2}+w_1 - \alpha_2-1 \right)A_{\alpha+\vec{e}_2} = -\left( \frac{n}{2} + w - 2k + \alpha_1 + 1 \right)A_{\alpha+\vec{e}_1} , \\
            \MoveEqLeft[10] \left( \frac{n}{2}+w_2-\alpha_4-1 \right)A_{\alpha+\vec{e}_4} = -\left( \frac{n}{2} + w - 2k + \alpha_1 + 1 \right)A_{\alpha+\vec{e}_1} \\
             & - \left(\frac{n}{2} + w_2 - 2\ell - \alpha_3 - 2\alpha_4 - 1 \right)A_{\alpha+\vec{e}_3} ,
        \end{split}
    \end{equation}
    for $w := w_1 + w_2$.
    If $\cI$ is a scalar Riemannian invariant of weight $-2\ell$ on $(n+2)$-manifolds, then
    \begin{equation}
        \label{eqn:inner-ovsienko-redou}
        \cD(\cu \otimes \cv) := \sum_{\alpha\in\sI_{k-\ell}^4} \binom{k-\ell}{\alpha} A_\alpha \cDelta^{\alpha_1} \left( \bigl( \cDelta^{\alpha_2}\cu \bigr) \bigl( \cDelta^{\alpha_3}( \cI\cDelta^{\alpha_4}\cv) \bigr) \right)
    \end{equation}
    is a tangential ambient bidifferential operator on $\cmE[w_1] \otimes \cmE[w_2]$.
    Moreover, the space of solutions of the recurrence relations~\eqref{eqn:inner-ovsienko-redou-recurrence} is at least $(k-\ell+1)$-dimensional.
\end{proposition}

\begin{proof}
    Define $\cD$ by Equation~\eqref{eqn:inner-ovsienko-redou}.
    Direct computation yields
    \begin{equation*}
        \bigl[ \cD , \cQ \bigr]_j = -4(k-\ell)\sum_{\alpha \in \sI_{k-\ell-1}^4} \binom{k-\ell-1}{\alpha} B_\alpha^{(j)} \cDelta^{\alpha_1} \left( \bigl( \cDelta^{\alpha_2}\cu \bigr) \bigl( \cDelta^{\alpha_3}( \cI \cDelta^{\alpha_4}\cv ) \bigr) \right)
    \end{equation*}
    on $\cmE[w_1-2] \otimes \cmE[w_2]$ when $j=1$, and on $\cmE[w_1] \otimes \cmE[w_2-2]$ when $j=2$, where
    \begin{align*}
        B_\alpha^{(1)} & := \left( \frac{n}{2} + w_1 - \alpha_2 - 1 \right)A_{\alpha+\vec{e}_2} + \left( \frac{n}{2} + w - 2k + \alpha_1 + 1 \right)A_{\alpha+\vec{e}_1} , \\
        B_\alpha^{(2)} & := \left( \frac{n}{2} + w_2 - \alpha_4 - 1 \right)A_{\alpha+\vec{e}_4} + \left( \frac{n}{2} + w - 2k + \alpha_1 + 1 \right)A_{\alpha+\vec{e}_1} \\
         & \quad + \left( \frac{n}{2} + w_2 - 2\ell - \alpha_3 - 2\alpha_4 - 1 \right)A_{\alpha+\vec{e}_3} .
    \end{align*}
    It follows immediately from \cref{what-to-check-for-tangential} and our assumptions on $A$ that $\cD$ is tangential on $\cmE[w_1] \otimes \cmE[w_2]$.

    Given $\boldsymbol{w} \in \bR^2$, define $F_{j,\boldsymbol{w}} \colon \sF_s^4 \to \sF_{s-1}^4$, $j \in \{1,2,3,4\}$ and $s \in \bN$, by
    \begin{align*}
        F_{1,\boldsymbol{w}}(A)_\alpha & := \left( \frac{n}{2} + w - 2k + \alpha_1 + 1 \right)A_{\alpha + \vec{e}_1} , \\
        F_{2,\boldsymbol{w}}(A)_\alpha & := \left( \frac{n}{2} + w_1 - \alpha_2 -1 \right)A_{\alpha + \vec{e}_2} , \\
        F_{3,\boldsymbol{w}}(A)_\alpha & := \left( \frac{n}{2} + w_2 - 2\ell - \alpha_3 - 2\alpha_4 - 1 \right) , \\
        F_{4,\boldsymbol{w}}(A)_\alpha & := \left( \frac{n}{2} + w_2 - \alpha_4 - 1 \right)A_{\alpha + \vec{e}_4} .
    \end{align*}
    Direct computation implies that
    \begin{equation*}
        0 \longrightarrow \sF_{k-\ell}^4 \overset{d_{1,\boldsymbol{w}}}{\longrightarrow} \bigl(\sF_{k-\ell-1}^4\bigr)^2 \overset{d_{2,\boldsymbol{w}}}{\longrightarrow} \sF_{k-\ell-2}^4 \longrightarrow 0
    \end{equation*}
    is a chain complex, where
    \begin{align*}
        d_{1,\boldsymbol{w}} & := \begin{pmatrix} F_{1,\boldsymbol{w}} + F_{2,\boldsymbol{w}} & F_{1,\boldsymbol{w}} + F_{3,\boldsymbol{w}} + F_{4,\boldsymbol{w}} \end{pmatrix}^T , \\
        d_{2,\boldsymbol{w}} & := \begin{pmatrix} F_{1,\boldsymbol{w}} + F_{3,\boldsymbol{w}} + F_{4,\boldsymbol{w}} & -( F_{1,\boldsymbol{w}} + F_{2,\boldsymbol{w}}) \end{pmatrix} .
    \end{align*}
    Note that each element of $\ker d_{1,\boldsymbol{w}}$ gives a solution of the recurrence relation~\eqref{eqn:inner-ovsienko-redou-recurrence}.
    An argument as in the proof of \cref{tridifferential} implies that
    \begin{equation*}
        \dim \ker d_{1,\boldsymbol{w}} \geq \lv \sF_{k-\ell}^4 \rv - 2\lv\sF_{k-\ell-1}^4\rv + \lv \sF_{k-\ell-2}^4\rv = k-\ell+1 . \qedhere
    \end{equation*}
\end{proof}

Our third bidifferential operator modifies a curved Ovsienko--Redou operator by replacing the outer Laplacian by a sum $\sum \cDelta^{\alpha_1} \circ I \circ \cDelta^{\alpha_2}$:

\begin{proposition}
    \label{inner-ovsienko-redou2}
    Fix integers $k \geq \ell \geq 0$ and $n \geq 3$;
    if $n$ is even, then assume additionally that $n \geq 2k$.
    Fix $w_1,w_2 \in \bR$ and let $A \in \sF_{k-\ell}^4$ be a solution of the recurrence relations
    \begin{equation}
        \label{eqn:inner-ovsienko-redou-recurrence2}
        \begin{split}
            \MoveEqLeft[10] \left( \frac{n}{2}+w_1-\alpha_3-1 \right)A_{\alpha+\vec{e}_3} = -\left( \frac{n}{2} + w - 2k + \alpha_1 + 1 \right)A_{\alpha+\vec{e}_1} \\
             & - \left(\frac{n}{2} + w - \alpha_2 - 2\alpha_3 - 2\alpha_4 - 1 \right)A_{\alpha+\vec{e}_2} , \\
            \MoveEqLeft[10] \left( \frac{n}{2}+w_2-\alpha_4-1 \right)A_{\alpha+\vec{e}_4} = -\left( \frac{n}{2} + w - 2k + \alpha_1 + 1 \right)A_{\alpha+\vec{e}_1} \\
             & - \left(\frac{n}{2} + w - \alpha_2 - 2\alpha_3 - 2\alpha_4 - 1 \right)A_{\alpha+\vec{e}_2} ,
        \end{split}
    \end{equation}
    for $w := w_1 + w_2$.
    If $\cI$ is a scalar Riemannian invariant of weight $-2\ell$ on $(n+2)$-manifolds, then
    \begin{equation*}
        \label{eqn:inner-ovsienko-redou2}
        \cD(\cu \otimes \cv) := \sum_{\alpha\in\sI_{k-\ell}^4} \binom{k-\ell}{\alpha} A_\alpha \cDelta^{\alpha_1} \left( \cI \cDelta^{\alpha_2} \bigl( ( \cDelta^{\alpha_3}\cu ) ( \cDelta^{\alpha_4}\cv ) \bigr) \right)
    \end{equation*}
    is a tangential ambient bidifferential operator on $\cmE[w_1] \otimes \cmE[w_2]$.
    Moreover, the space of solutions of the recurrence relations~\eqref{eqn:inner-ovsienko-redou-recurrence2} is at least $(k-\ell+1)$-dimensional.
\end{proposition}

\begin{proof}
    Define $\cD$ by Equation~\eqref{eqn:inner-ovsienko-redou2}.
    Direct computation yields
    \begin{equation*}
        \bigl[ \cD , \cQ \bigr]_j = -4(k-\ell)\sum_{\alpha \in \sI_{k-\ell-1}^4} \binom{k-\ell-1}{\alpha} B_\alpha^{(j)} \cDelta^{\alpha_1} \left( \cI \cDelta^{\alpha_2} \bigl( ( \cDelta^{\alpha_3}\cu ) ( \cDelta^{\alpha_4}\cv ) \bigr) \right)
    \end{equation*}
    on $\cmE[w_1-2] \otimes \cmE[w_2]$ when $j=1$, and on $\cmE[w_1] \otimes \cmE[w_2-2]$ when $j=2$, where
    \begin{align*}
        B_\alpha^{(1)} & := \left( \frac{n}{2} + w_1 - \alpha_3 - 1 \right)A_{\alpha+\vec{e}_3} + \left( \frac{n}{2} + w - 2k + \alpha_1 + 1 \right)A_{\alpha+\vec{e}_1} \\
         & \quad + \left( \frac{n}{2} + w - \alpha_2 - 2\alpha_3 - 2\alpha_4 - 1 \right)A_{\alpha+\vec{e}_2} , \\
        B_\alpha^{(2)} & := \left( \frac{n}{2} + w_2 - \alpha_4 - 1 \right)A_{\alpha+\vec{e}_4} + \left( \frac{n}{2} + w - 2k + \alpha_1 + 1 \right)A_{\alpha+\vec{e}_1} \\
         & \quad + \left( \frac{n}{2} + w - \alpha_2 - 2\alpha_3 - 2\alpha_4 - 1 \right)A_{\alpha+\vec{e}_2} .
    \end{align*}
    It follows immediately from \cref{what-to-check-for-tangential} and our assumptions on $A$ that $\cD$ is tangential on $\cmE[w_1] \otimes \cmE[w_2]$.

    Given $\boldsymbol{w} \in \bR^2$, define $F_{j,\boldsymbol{w}},F_{2,\boldsymbol{w}}' \colon \sF_s^4 \to \sF_{s-1}^4$, $j \in \{ 1,2,3,4 \}$ and $s \in \bN$, by
    \begin{align*}
        F_{1,\boldsymbol{w}}(A)_\alpha & := \left( \frac{n}{2} + w - 2k + \alpha_1 + 1 \right)A_{\alpha + \vec{e}_1} , \\
        F_{2,\boldsymbol{w}}(A)_\alpha & := \left( \frac{n}{2} + w - \alpha_2 - 2\alpha_3 - 2\alpha_4 - 1 \right)A_{\alpha + \vec{e}_2} , \\
        F_{2,\boldsymbol{w}}'(A)_\alpha & := \left( \frac{n}{2} + w - \alpha_2 - 2\alpha_3 - 2\alpha_4 - 3 \right)A_{\alpha + \vec{e}_2} , \\
        F_{3,\boldsymbol{w}}(A)_\alpha & := \left(\frac{n}{2} + w_1 - \alpha_3 - 1 \right)A_{\alpha + \vec{e}_3} , \\
        F_{4,\boldsymbol{w}}(A)_\alpha & := \left( \frac{n}{2} + w_2 - \alpha_4 - 1 \right)A_{\alpha + \vec{e}_4} .
    \end{align*}
    Direct computation implies that
    \begin{equation*}
        0 \longrightarrow \sF_{k-\ell}^4 \overset{d_{1,\boldsymbol{w}}}{\longrightarrow} \bigl(\sF_{k-\ell-1}^4\bigr)^2 \overset{d_{2,\boldsymbol{w}}}{\longrightarrow} \sF_{k-\ell-2}^4 \longrightarrow 0
    \end{equation*} is a chain complex, where
    \begin{align*}
        d_{1,\boldsymbol{w}} & := \begin{pmatrix} F_{1,\boldsymbol{w}} + F_{2,\boldsymbol{w}} + F_{3,\boldsymbol{w}} & F_{1,\boldsymbol{w}} + F_{2,\boldsymbol{w}} + F_{4,\boldsymbol{w}} \end{pmatrix}^T , \\
        d_{2,\boldsymbol{w}} & := \begin{pmatrix} F_{1,\boldsymbol{w}} + F_{2,\boldsymbol{w}}' + F_{4,\boldsymbol{w}} & -( F_{1,\boldsymbol{w}} + F_{2,\boldsymbol{w}}' + F_{3,\boldsymbol{w}} ) \end{pmatrix} .
    \end{align*}
    Note that each element of $\ker d_{1,\boldsymbol{w}}$ gives a solution of recurrence relation~\eqref{eqn:inner-ovsienko-redou-recurrence2}.
    An argument as in the proof of \cref{tridifferential} implies that
    \begin{equation*}
        \dim \ker d_{1,\boldsymbol{w}} \geq \lv \sF_{k-\ell}^4 \rv - 2\lv\sF_{k-\ell-1}^4\rv + \lv\sF_{k-\ell-2}^4\rv = k-\ell+1 . \qedhere
    \end{equation*}
\end{proof}
\section{Formal self-adjointness in the ambient space}
\label{sec:fsa}

In this section we prove \cref{fsa}. First, we show that the recurrence relations in the proof of \cref{tridifferential} force certain identities in the coefficients.
Then we use those identities to show that a symmetric tangential operator must be formally self-adjoint.
The identities are as follows:

\begin{lemma}
    \label{index-permutation}
    Let $n > 2k$ be integers and let $A \in \sF^k_5$ be a solution to the recurrence relations
    \begin{subequations}
        \label{eqn:recurrence}
        \begin{align}
            \label{eqn:recurrence-1} (n+2k-4\alpha_3-4)A_{\alpha+\vec{e}_3} & = (n+2k-4\alpha_4-4)A_{\alpha+\vec{e}_4} , \\
            \label{eqn:recurrence-2} (n+2k-4\alpha_5-4)A_{\alpha+\vec{e}_5} & = (n+2k-4\alpha_1-4)A_{\alpha+\vec{e}_1} , \\
            \label{eqn:recurrence-3} (n+2k-4\alpha_3-4)A_{\alpha+\vec{e}_3} & = (n+2k-4\alpha_1-4)A_{\alpha + \vec{e}_1} , \\
               \notag & \quad - 4(\alpha_1 - \alpha_3 - \alpha_4 + \alpha_5)A_{\alpha + \vec{e_2}} ,
        \end{align}
    \end{subequations}
    for all $\alpha \in \sI_{k-1}^5$. 
    Then
    \begin{align}
        \label{eqn:permutation1} A_{\alpha_1, \alpha_2, \alpha_3, \alpha_4, \alpha_5} & = A_{\alpha_1, \alpha_2, \alpha_4, \alpha_3, \alpha_5} , \\
        \label{eqn:permutation2} A_{\alpha_1, \alpha_2, \alpha_3, \alpha_4, \alpha_5} & = A_{\alpha_5, \alpha_2, \alpha_3, \alpha_4, \alpha_1} , \\
        \label{eqn:permutation3} A_{\alpha_1, \alpha_2, \alpha_3, \alpha_4, \alpha_5} & = A_{\alpha_3, \alpha_2, \alpha_1, \alpha_5, \alpha_4} ,
    \end{align}
    for all $\alpha \in \sI_k^5$.
\end{lemma}

The proof of Lemma \ref{index-permutation} is by a long induction argument. Its key ingredient is the following elementary result:
\begin{lemma}
    \label{induction-lemma}
    Let $A \in \sF^2_{k+1}$ and $f \colon \bN_0 \to \bR$ be such that $f$ is nowhere zero and
    \begin{equation*}
        f(\alpha_1+1)A_{\alpha_1+1,\alpha_2} = f(\alpha_2+1)A_{\alpha_1,\alpha_2+1}
    \end{equation*}
    for all $\alpha \in \sI_k^2$.
    Then $A_{\alpha_1,\alpha_2} = A_{\alpha_2,\alpha_1}$ for all $\alpha\in\sI^2_{k+1}$.
\end{lemma}

\begin{proof}
    Without loss of generality, assume $\alpha_1 > \alpha_2$.
    We observe that
    \begin{equation*}
        A_{\alpha_1, \alpha_2}=\frac{f(\alpha_2+1)}{f(\alpha_1)}A_{\alpha_1-1, \alpha_2+1} .
    \end{equation*}
    Iterating this yields
    \begin{equation*}
        A_{\alpha_1, \alpha_2}=\frac{\prod_{i=1}^{\alpha_1-\alpha_2} f(\alpha_2+i)}{\prod_{i=1}^{\alpha_1-\alpha_2} f(\alpha_1+1-i)} A_{\alpha_2, \alpha_1}
    \end{equation*}
    The conclusion follows by reindexing one of the products.
\end{proof}

As we will see below, Lemma \ref{induction-lemma} immediately implies the first two conclusions of Lemma \ref{index-permutation}.
The last conclusion follows from a double induction argument, the first step of which handles the case $\alpha_4=\alpha_5$.
To that end, given $\alpha \in \sI_5^k$, set
\begin{equation*}
    \alpha' := (\alpha_3, \alpha_2, \alpha_1, \alpha_5, \alpha_4) \in \sI_k^5 .
\end{equation*}

\begin{lemma}
    \label{d-e=0 lemma}
    Let $n > 2k$ be integers and let $A \in \sF^5_k$ satisfy recurrence relations \eqref{eqn:recurrence}.
    Let $\alpha \in \sI_k^5$ be such that $\alpha_4=\alpha_5$.
    Then
    \begin{equation*}
        A_{\alpha}=A_{\alpha'} .
    \end{equation*}
\end{lemma}

\begin{proof}
    Let $\alpha \in \sI_{k-1}^5$ be such that $\alpha_4=\alpha_5$.
    Applying Equation~\eqref{eqn:recurrence-3} to $\alpha'$ and $\alpha$ yields
    \begin{equation*}
        \begin{split}
            0 & = (n+2k-4\alpha_1-4)A_{\alpha'+\vec{e}_3}-4(\alpha_1-\alpha_3)A_{\alpha'+\vec{e}_2}-(n+2k-4\alpha_3-4)A_{\alpha'+\vec{e}_1} , \\
            0 & = (n+2k-4\alpha_1-4)A_{\alpha+\vec{e}_1}-4(\alpha_1-\alpha_3)A_{\alpha+\vec{e}_2}-(n+2k-4\alpha_3-4)A_{\alpha+\vec{e}_3} .
        \end{split}
    \end{equation*}
    Suppose additionally that $A_{\alpha+\vec{e}_2}=A_{\alpha'+\vec{e}_2}$ for all $\alpha\in\sI^5_{k-1}$. Then
    \begin{equation*}
        (n+2k-4\alpha_1-4)(A_{\alpha+\vec{e}_1}-A_{\alpha'+\vec{e}_3}) = (n+2k-4\alpha_3-4)(A_{\alpha+\vec{e}_3}-A_{\alpha'+\vec{e}_1}) .
    \end{equation*}
    Set $B_{\alpha_1, \alpha_3} := A_{\alpha_1,\alpha_2, \alpha_3, \alpha_4, \alpha_4}-A_{\alpha_3, \alpha_2, \alpha_1, \alpha_4, \alpha_4}$.
    It follows from Lemma \ref{induction-lemma} and the above display that $B$ is symmetric.
    Since $B$ is manifestly antisymmetric, it must vanish.
    The final conclusion follows by induction.
\end{proof}

\begin{proof}[Proof of \cref{index-permutation}]
    Applying \cref{induction-lemma} to Equations~\eqref{eqn:recurrence-1} and~\eqref{eqn:recurrence-2} yields Equations~\eqref{eqn:permutation1} and~\eqref{eqn:permutation2}, respectively.
    
    Equations~\eqref{eqn:recurrence-1} and~\eqref{eqn:recurrence-2} also imply that
    \begin{equation*}
        (n+2k-4\alpha_5-4)(A_{\alpha+\vec{e}_5}-A_{\alpha'+\vec{e}_4}) = (n+2k-4\alpha_1-4)(A_{\alpha+\vec{e}_1}-A_{\alpha'+\vec{e}_3}) .
    \end{equation*}
    Thus, $A_{\alpha+\vec{e}_1}=A_{\alpha'+\vec{e}_3}$ if and only if $A_{\alpha+\vec{e}_5}=A_{\alpha'+\vec{e}_4}$. 

    Without loss of generality, let $\alpha_5\geq \alpha_4$.
    Suppose that $A_{\alpha}=A_{\alpha'}$ for all $\alpha$ such that $\alpha_5-\alpha_4=m\in \bN_0$.
    The previous paragraph implies that $A_{\alpha}=A_{\alpha'}$ for all $\alpha$ such that $\alpha_5-\alpha_4=m+1$.
    We conclude from \cref{d-e=0 lemma} and induction that $A_{\alpha}=A_{\alpha'}$.
\end{proof}

We now use \cref{index-permutation} to prove the formal self-adjointness of $\cD'$.

\begin{proof}[Proof of \cref{fsa}]
    By the form of $\cD'$, it suffices to show that
    \begin{align}
        \label{eqn:step1} \int_{\cmG} \cx \cD(\cu \otimes \cv \otimes \cw) \dvol & =\int_{\cmG} \cu \cD(\cx \otimes \cw \otimes \cv) \dvol , \\
        \label{eqn:step2} \int_{\cmG} \cx \cD(\cw \otimes \cu \otimes \cv) \dvol & =\int_{\cmG} \cu \cD(\cv \otimes \cx \otimes \cw) \dvol , \\
        \label{eqn:step3} \int_{\cmG} \cx \cD(\cv \otimes \cw \otimes \cu) \dvol & =\int_{\cmG} \cu \cD(\cw \otimes \cv \otimes \cx ) \dvol ,
    \end{align}
    for all $\cu, \cv, \cw \in C_0^{\infty}(\cmG)$.

    Combining Equation~\eqref{eqn:permutation3} and the formal self-adjointness of $\cDelta$ yield
    \begin{align*}
        \int_{\cmG} \cx \cD(\cu \otimes \cv \otimes\cw) \dvol & = \sum_{\alpha \in \sI_k^5} \int_{\cmG} \binom{k}{\alpha}A_{\alpha} (\cDelta^{\alpha_1}\cx)(\cDelta^{\alpha_5}\cw)\cDelta^{\alpha_2}\bigl((\cDelta^{\alpha_3}\cu)(\cDelta^{\alpha_4}\cv)\bigr) \dvol \\
        & = \sum_{\alpha \in \sI_k^5} \int_{\cmG} \binom{k}{\alpha'}A_{\alpha'} (\cDelta^{\alpha_1}\cu)(\cDelta^{\alpha_5}\cv)\cDelta^{\alpha_2}\bigl( (\cDelta^{\alpha_3}\cx)(\cDelta^{\alpha_4}\cw) \bigr) \dvol \\
        & = \int_{\cmG} \cu \cD( \cx \otimes \cw \otimes \cv ) \dvol
    \end{align*}
    for all $\cu,\cv,\cw \in C_0^\infty(\cmG)$.
    Thus Equation~\eqref{eqn:step1} holds.
    Equation~\eqref{eqn:permutation1} immediately implies that $\cD(\cu \otimes \cv \otimes \cw) = \cD(\cv \otimes \cu \otimes \cw)$ for all $\cu,\cv,\cw \in C^\infty(\cmG)$.
    Combining this with Equation~\eqref{eqn:step1} yields Equation~\eqref{eqn:step2}.
    Finally, Equation~\eqref{eqn:permutation2} implies that
    \begin{equation*}
        \begin{split}
            \int_{\cmG} \cx \cD(\cv \otimes \cw \otimes\cu) \dvol & = \sum_{\alpha\in\sI^5_k}\int_{\cmG} \binom{k}{\alpha} A_{\alpha}  (\cDelta^{\alpha_1}\cx)(\cDelta^{\alpha_5}\cu)\cDelta^{\alpha_2}\bigl( (\cDelta^{\alpha_3}\cv)(\cDelta^{\alpha_4}\cw)\bigr) \dvol \\
            & = \sum_{\alpha\in\sI^5_k}\int_{\cmG} \binom{k}{\alpha} A_{\alpha}  (\cDelta^{\alpha_1}\cu)(\cDelta^{\alpha_5}\cx)\cDelta^{\alpha_2}\bigl( (\cDelta^{\alpha_3}\cv)(\cDelta^{\alpha_4}\cw)\bigr) \dvol \\
            & = \int_{\cmG} \cu \cD(\cv \otimes \cw \otimes \cx ) \dvol
        \end{split}
    \end{equation*}
    for all $\cu,\cv,\cw \in C_0^\infty(\cmG)$.
    Equation~\eqref{eqn:step3} now follows from Equation~\eqref{eqn:permutation1}.
\end{proof}

\begin{remark}
    \label{rk:dimension}
    A straightforward computation verifies that if $(M,g)$ is a Riemannian manifold and $u_1,u_2,u_3 \in C^\infty(M)$, then
    \begin{multline*}
        \Delta(u_1u_2u_2) + u_2u_3\Delta u_1 + u_1u_3\Delta u_2 + u_1u_2\Delta u_3 \\ = u_1\Delta(u_2u_3) + u_2\Delta(u_1u_3) + u_3\Delta(u_1u_2) .
    \end{multline*}
    Note that both sides of this identity define formally self-adjoint operators.
    One can obtain further identities relating formally self-adjoint tridifferential operators of order $2k$ by taking Laplacians of both sides and replacing $u_j$ by suitable Laplace powers thereof.
    This leads to one linear relation among the summands in \cref{fsa}, and hence our guess that the space of formally self-adjoint, conformally covariant, tridifferential operators of order $2k$ on the standard conformal $n$-sphere is $k$-dimensional.
\end{remark}

We conclude by discussing the special cases of \cref{linear-analogue,outer-ovsienko-redou,inner-ovsienko-redou} when the weights satisfy the necessary condition for formal self-adjointness.
Like \cref{fsa}, formal self-adjointness of the symmetrized ambient operator follows immediately from tangentiality.

The formally self-adjoint special case of \cref{linear-analogue} is as follows:

\begin{proposition}
    \label{linear-analogue-fsa}
    Fix integers $n \geq 3$ and $k,\ell_1,\ell_2 \geq 0$ such that $k \geq \ell := \ell_1 + \ell_2$;
    if $n$ is even, then assume additionally that $n \geq 2k$.
    Let $A \in \sF_{k-\ell}^3$ be a solution of the recurrence relation
    \begin{equation*}
        (k - \alpha_3 - 1)A_{\alpha+\vec{e}_3} = (k - \alpha_1-1)A_{\alpha+\vec{e}_1} - ( \ell_1 - \ell_2 + \alpha_1 - \alpha_3 ) A_{\alpha+\vec{e}_2} .
    \end{equation*}
    If $\cI_1,\cI_2$ are scalar Riemannian invariants of weight $-2\ell_1$ and $-2\ell_2$, respectively, on $(n+2)$-manifolds, then
    \begin{equation}
        \label{eqn:linear-analogue-fsa}
        \cD(\cu) := \sum_{\alpha\in\sI_{k-\ell}^3} \binom{k-\ell}{\alpha}A_\alpha \Bigl( \cDelta^{\alpha_1}\bigl( \cI_1\cDelta^{\alpha_2}( \cI_2\cDelta^{\alpha_3}\cu) + \cDelta^{\alpha_3}\bigl( \cI_2\cDelta^{\alpha_2}( \cI_1\cDelta^{\alpha_1}\cu ) \bigr) \Bigr)
    \end{equation}
    is a tangential formally self-adjoint ambient differential operator on $\cmE\bigl[-\frac{n-2k}{2}\bigr]$.
\end{proposition}

\begin{proof}
    Define $\cD$ by Equation~\eqref{eqn:linear-analogue-fsa}.
    \Cref{linear-analogue} implies that $\cD$ is tangential on $\cmE\bigl[-\frac{n-2k}{2}\bigr]$.
    The Divergence Theorem implies that
    \begin{multline*}
        \int_{\cmG} \cu_1\cD(\cu_2) \dvol_{\cg} = \sum_{\alpha\in\sI_{k-\ell}^3} \int_{\cmG} \binom{k-\ell}{\alpha}A_\alpha \Bigl( \bigl(\cI_1\cDelta^{\alpha_1}\cu_1\bigr) \bigl( \cDelta^{\alpha_2}( \cI_2\cDelta^{\alpha_3}\cu_2)\bigr) \\ + \bigl( \cI_1\cDelta^{\alpha_1}\cu_2 \bigr) \bigl( \cDelta^{\alpha_2}( \cI_2\cDelta^{\alpha_3}\cu_1 ) \bigr) \Bigr) \dvol_{\cg}
    \end{multline*}
    for all $\cu_1,\cu_2 \in C_0^\infty(\cmG)$.
    Hence $\cD$ is formally self-adjoint.
\end{proof}

The formally self-adjoint case of \cref{outer-ovsienko-redou} is as follows:

\begin{proposition}
    \label{outer-ovsienko-redou-fsa}
    Fix integers $k \geq \ell \geq 0$ and $n > 2k$.
    Let $A \in \sF_{k-\ell}^3$ be a solution of the recurrence relations
    \begin{align*}
        \left( \frac{n+4k}{6} - \alpha_1 - 1 \right)A_{\alpha+\vec{e}_1} & = \left(\frac{n+4k}{6} - \alpha_2 - 1 \right)A_{\alpha+\vec{e}_2} , \\
        \left( \frac{n+4k}{6} - \alpha_1 - 1 \right)A_{\alpha+\vec{e}_1} & = \left( \frac{n+4k}{6} - \alpha_3 - 1 \right)A_{\alpha+\vec{e}_3} .
    \end{align*}
    If $\cI$ is a scalar Riemannian invariant of weight $-2\ell$ on $(n+2)$-manifolds, then
    \begin{equation}
        \label{eqn:outer-ovsienko-redou-fsa}
        \cD(\cu \otimes \cv) := \sum_{\alpha \in \sI_{k-\ell}^3} \binom{k-\ell}{\alpha}A_{\alpha}\cDelta^{\alpha_1}\left( \cI \bigl( \cDelta^{\alpha_2}\cu \bigr) \bigl( \cDelta^{\alpha_3}\cv \bigr) \right)
    \end{equation}
    is a tangential formally self-adjoint ambient bidifferential operator on $\cmE\bigl[-\frac{n-2k}{3}\bigr]^{\otimes 2}$.
\end{proposition}

\begin{remark}
    \label{rk:why-strict-inequality}
    The assumption $n > 2k$ guarantees that $0$ is not a coefficient in the assumed recurrence relations, allowing us to apply \cref{induction-lemma} to deduce symmetry.
    A similar argument works in the case $n=2k$, where $\cD$ is still defined, provided one also assumes that $A_{k\vec{e}_1}=A_{k\vec{e}_2}=A_{k\vec{e}_3}$.
\end{remark}

\begin{proof}
    Define $\cD$ by Equation~\eqref{eqn:outer-ovsienko-redou-fsa}.
    \Cref{outer-ovsienko-redou} implies that $\cD$ is tangential on $\cmE\bigl[-\frac{n-2k}{3}\bigr]^{\otimes2}$.
    The Divergence Theorem implies that
    \begin{equation*}
        \int_{\cmG} \cu_1\cD(\cu_2 \otimes \cu_3) \dvol_{\cg} = \sum_{\alpha \in \sI_{k-\ell}^3} \int_{\cmG} \binom{k-\ell}{\alpha} A_\alpha \cI \bigl( \cDelta^{\alpha_1}\cu_1 \bigr) \bigl( \cDelta^{\alpha_2}\cu_2 \bigr) \bigl( \cDelta^{\alpha_3}\cu_3 \bigr) \dvol_{\cg}
    \end{equation*}
    for all $\cu_1,\cu_2,\cu_3 \in C_0^\infty(\cmG)$.
    \Cref{induction-lemma} and our recurrence relations imply that $A$ is symmetric, and hence $\cD$ is formally self-adjoint.
\end{proof}

The formally self-adjoint case of \cref{inner-ovsienko-redou} is as follows:

\begin{proposition}
    \label{inner-ovsienko-redou-fsa}
    Fix integers $k \geq \ell \geq 0$ and $n > 2k$.
    Let $A \in \sF_{k-\ell}^4$ be a solution of the recurrence relations
    \begin{align*}
        \MoveEqLeft[10] \left( \frac{n+4k}{6} - \alpha_3 - 1 \right)A_{\alpha+\vec{e}_3} = \left( \frac{n+4k}{6} - \alpha_4 - 1 \right)A_{\alpha+\vec{e}_4} , \\
        \MoveEqLeft[10] \left( \frac{n+4k}{6} - \alpha_3 - 1 \right)A_{\alpha+\vec{e}_3} = \left( \frac{n+4k}{6} - \alpha_1 - 1 \right)A_{\alpha+\vec{e}_1} \\
        & + \left(\frac{n+4k}{6}-2\ell-2\alpha_1-\alpha_2-1 \right)A_{\alpha+\vec{e}_2} .
    \end{align*}
    If $\cI$ is a scalar Riemannian invariant of weight $-2\ell$ on $(n+2)$-manifolds, then
    \begin{multline}
        \label{eqn:inner-ovsienko-redou-fsa}
        \cD( \cu \otimes \cv ) := \sum_{\alpha \in \sI_{k-\ell}^4} \binom{k-\ell}{\alpha} A_{\alpha} \biggl[ \cDelta^{\alpha_1} \left( \cI \cDelta^{\alpha_2} \bigl( (\cDelta^{\alpha_3}\cu) (\cDelta^{\alpha_4}\cv) \bigr) \right) \\
            + \cDelta^{\alpha_3} \left( \bigl(\cDelta^{\alpha_4}\cu \bigr) \bigl( \cDelta^{\alpha_2} ( \cI \cDelta^{\alpha_1}\cv ) \bigr) + \bigl(\cDelta^{\alpha_4}\cv \bigr) \bigl( \cDelta^{\alpha_2} ( \cI \cDelta^{\alpha_1}\cu ) \bigr) \right) \biggr]
    \end{multline}
    is a tangential formally self-adjoint ambient bidifferential operator on $\cmE\bigl[ -\frac{n-2k}{3}\bigr]^{\otimes 2}$.
\end{proposition}

\begin{remark}
    As in \cref{rk:why-strict-inequality}, the assumption $n>2k$ is made in order to apply \cref{induction-lemma}.
    The result is also true when $n=2k$ provided one also assumes that $A_{k\vec{e}_3}=A_{k\vec{e}_4}$.
\end{remark}

\begin{proof}
    Define $\cD$ by Equation~\eqref{eqn:inner-ovsienko-redou-fsa}.
    \Cref{inner-ovsienko-redou,inner-ovsienko-redou2} imply that $\cD$ is tangential on $\cmE\bigl[-\frac{n-2k}{3}\bigr]^{\otimes2}$.
    The Divergence Theorem implies that
    \begin{align*}
        \MoveEqLeft[15] \int_{\cmG} \cu_1\cD(\cu_2\otimes\cu_3) \dvol_{\cg} = \sum_{\alpha\in\sI_{k-\ell}^4} \int_{\cmG} \binom{k-\ell}{\alpha}A_\alpha \Bigl( \bigl( \cI\cDelta^{\alpha_1}\cu_1 \bigr) \cDelta^{\alpha_2}\bigl( ( \cDelta^{\alpha_3}\cu_2 ) ( \cDelta^{\alpha_4}\cu_3 ) \bigr) \\
        & + \bigl( \cI\cDelta^{\alpha_1}\cu_2 \bigr) \cDelta^{\alpha_2}\bigl( (\cDelta^{\alpha_3}\cu_1) (\cDelta^{\alpha_4}\cu_3) \bigr) \\
        & + \bigl( \cI\cDelta^{\alpha_1}\cu_3\bigr) \cDelta^{\alpha_2}\bigl( ( \cDelta^{\alpha_3}\cu_1) (\cDelta^{\alpha_4}\cu_2) \bigr) \Bigr) \dvol_{\cg}
    \end{align*}
    for all $\cu_1,\cu_2,\cu_3 \in C_0^\infty(\cmG)$.
    Our first recurrence relation and \cref{induction-lemma} imply that $A$ is symmetric in its last two components.
    Hence $\cD$ is formally self-adjoint.
\end{proof}

\section*{Acknowledgements}
JSC was partially supported by a Simons Foundation Collaboration Grant for Mathematicians and by the National Science Foundation under Award No.\ DMS-2505606.
OC was partially supported through the WIM Scholar program at Penn State University.

\bibliography{bib}

@preamble { "\newcommand{\noopsort}[1]{} " }

@article {Alexakis2003,
    AUTHOR = {Alexakis, Spyros},
     TITLE = {On conformally invariant differential operators in odd
              dimensions},
   JOURNAL = {Proc. Natl. Acad. Sci. USA},
  FJOURNAL = {Proceedings of the National Academy of Sciences of the United
              States of America},
    VOLUME = {100},
      YEAR = {2003},
    NUMBER = {8},
     PAGES = {4409--4410},
      ISSN = {0027-8424},
   MRCLASS = {58J60 (53A30 58J70)},
  MRNUMBER = {1971495},
MRREVIEWER = {Michael G. Eastwood},
       DOI = {10.1073/pnas.0430972100},
       URL = {https://doi.org/10.1073/pnas.0430972100},
}

@article {BaileyEastwoodGraham1994,
    AUTHOR = {Bailey, Toby N. and Eastwood, Michael G. and Graham, C. Robin},
     TITLE = {Invariant theory for conformal and {CR} geometry},
   JOURNAL = {Ann. of Math. (2)},
  FJOURNAL = {Annals of Mathematics. Second Series},
    VOLUME = {139},
      YEAR = {1994},
    NUMBER = {3},
     PAGES = {491--552},
      ISSN = {0003-486X},
   MRCLASS = {53A30 (32C16 53A55)},
  MRNUMBER = {1283869},
MRREVIEWER = {William M. McGovern},
       DOI = {10.2307/2118571},
       URL = {https://doi.org/10.2307/2118571},
}

@article {Branson1995,
    AUTHOR = {Branson, Thomas P.},
     TITLE = {Sharp inequalities, the functional determinant, and the
              complementary series},
   JOURNAL = {Trans. Amer. Math. Soc.},
  FJOURNAL = {Transactions of the American Mathematical Society},
    VOLUME = {347},
      YEAR = {1995},
    NUMBER = {10},
     PAGES = {3671--3742},
      ISSN = {0002-9947},
   MRCLASS = {58G26 (22E46 53A30)},
  MRNUMBER = {1316845},
MRREVIEWER = {Friedbert Pr\"{u}fer},
       DOI = {10.2307/2155203},
       URL = {https://doi.org/10.2307/2155203},
}

@article {BransonGover2005,
    AUTHOR = {Branson, Thomas P. and Gover, A. Rod},
     TITLE = {Conformally invariant operators, differential forms,
              cohomology and a generalisation of {$Q$}-curvature},
   JOURNAL = {Comm. Partial Differential Equations},
  FJOURNAL = {Communications in Partial Differential Equations},
    VOLUME = {30},
      YEAR = {2005},
    NUMBER = {10-12},
     PAGES = {1611--1669},
      ISSN = {0360-5302},
   MRCLASS = {58J60 (53A30)},
  MRNUMBER = {2182307},
MRREVIEWER = {Michael G. Eastwood},
       DOI = {10.1080/03605300500299943},
       URL = {https://doi.org/10.1080/03605300500299943},
}

@article {Case2019fl,
    AUTHOR = {Case, Jeffrey S.},
     TITLE = {The {F}rank-{L}ieb approach to sharp {S}obolev inequalities},
   JOURNAL = {Commun. Contemp. Math.},
  FJOURNAL = {Communications in Contemporary Mathematics},
    VOLUME = {23},
      YEAR = {2021},
    NUMBER = {3},
     PAGES = {Paper No. 2050015, 16},
      ISSN = {0219-1997},
   MRCLASS = {53C18 (46E35)},
  MRNUMBER = {4216420},
       DOI = {10.1142/S0219199720500157},
       URL = {https://doi.org/10.1142/S0219199720500157},
}

@article {CaseGrahamKuo2023,
    AUTHOR = {Case, Jeffrey S. and Graham, C Robin and Kuo, Tzu-Mo},
     TITLE = {Extrinsic {GJMS} operators for submanifolds},
   JOURNAL = {Rev. Mat. Iberoam.},
  FJOURNAL = {Revista Matem\'{a}tica Iberoamericana},
    VOLUME = {41},
      YEAR = {2025},
    NUMBER = {4},
     PAGES = {1393--1429},
      ISSN = {0213-2230},
   MRCLASS = {53A31 (53A10 53A55 53C18 58J70)},
  MRNUMBER = {4912924},
       DOI = {10.4171/rmi/1545},
       URL = {https://doi.org/10.4171/rmi/1545},
}

@article {CaseLinYuan2018b,
    AUTHOR = {Case, Jeffrey S. and Lin, Yueh-Ju and Yuan, Wei},
     TITLE = {Some constructions of formally self-adjoint conformally
              covariant polydifferential operators},
   JOURNAL = {Adv. Math.},
  FJOURNAL = {Advances in Mathematics},
    VOLUME = {401},
      YEAR = {2022},
     PAGES = {Paper No. 108312},
      ISSN = {0001-8708},
   MRCLASS = {53A31 (58J70)},
  MRNUMBER = {4392224},
       DOI = {10.1016/j.aim.2022.108312},
       URL = {https://doi.org/10.1016/j.aim.2022.108312},
}

@article {CaseLinYuan2022or,
    AUTHOR = {Case, Jeffrey S and Lin, Yueh-Ju and Yuan, Wei},
     TITLE = {Curved {V}ersions of the {O}vsienko--{R}edou {O}perators},
   JOURNAL = {Int. Math. Res. Not. IMRN},
  FJOURNAL = {International Mathematics Research Notices. IMRN},
      YEAR = {2023},
    NUMBER = {19},
     PAGES = {16904--16929},
      ISSN = {1073-7928,1687-0247},
   MRCLASS = {53C18 (58J40)},
  MRNUMBER = {4651902},
       DOI = {10.1093/imrn/rnad053},
       URL = {https://doi.org/10.1093/imrn/rnad053},
}

@unpublished {CaseYan2024,
    AUTHOR = {Case, Jeffrey S. and Yan, Zetian},
     TITLE = {Formal self-adjointness of a family of conformally invariant bidifferential operators},
      YEAR = {preprint},
      NOTE = {arXiv:2405.09532},
}

@unpublished {ChernYan2024,
    AUTHOR = {Chern, Shane and Yan, Zetian},
     TITLE = {Juhl type formulas for curved {Ovsienko}--{Redou} operators},
      YEAR = {\noopsort{2099}preprint},
      NOTE = {arXiv:2407.12280},
}

@article {Clerc2016,
    AUTHOR = {Clerc, Jean-Louis},
     TITLE = {Singular conformally invariant trilinear forms, {I}: {T}he
              multiplicity one theorem},
   JOURNAL = {Transform. Groups},
  FJOURNAL = {Transformation Groups},
    VOLUME = {21},
      YEAR = {2016},
    NUMBER = {3},
     PAGES = {619--652},
      ISSN = {1083-4362},
   MRCLASS = {22E46 (22E45 32C30 32M15 53A30)},
  MRNUMBER = {3531743},
MRREVIEWER = {Salah Mehdi},
       DOI = {10.1007/s00031-016-9365-x},
       URL = {https://doi.org/10.1007/s00031-016-9365-x},
}

@article {Clerc2017,
    AUTHOR = {Clerc, Jean-Louis},
     TITLE = {Singular conformally invariant trilinear forms, {II}: {T}he
              higher multiplicity case},
   JOURNAL = {Transform. Groups},
  FJOURNAL = {Transformation Groups},
    VOLUME = {22},
      YEAR = {2017},
    NUMBER = {3},
     PAGES = {651--706},
      ISSN = {1083-4362},
   MRCLASS = {22A22 (32M99 53A30)},
  MRNUMBER = {3682833},
       DOI = {10.1007/s00031-016-9404-7},
       URL = {https://doi.org/10.1007/s00031-016-9404-7},
}

@article {FeffermanGraham2002,
    AUTHOR = {Fefferman, Charles and Graham, C. Robin},
     TITLE = {{$Q$}-curvature and {Poincar\'e} metrics},
   JOURNAL = {Math. Res. Lett.},
  FJOURNAL = {Mathematical Research Letters},
    VOLUME = {9},
      YEAR = {2002},
    NUMBER = {2-3},
     PAGES = {139--151},
      ISSN = {1073-2780},
   MRCLASS = {53C21 (58J60)},
  MRNUMBER = {1909634},
MRREVIEWER = {A. Rod Gover},
}

@book {FeffermanGraham2012,
    AUTHOR = {Fefferman, Charles and Graham, C. Robin},
     TITLE = {The ambient metric},
    SERIES = {Annals of Mathematics Studies},
    VOLUME = {178},
 PUBLISHER = {Princeton University Press},
   ADDRESS = {Princeton, NJ},
      YEAR = {2012},
     PAGES = {x+113},
      ISBN = {978-0-691-15313-1},
   MRCLASS = {53Axx},
  MRNUMBER = {2858236},
}

@article {FeffermanGraham2013,
    AUTHOR = {Fefferman, Charles and Graham, C. Robin},
     TITLE = {Juhl's formulae for {GJMS} operators and {$Q$}-curvatures},
   JOURNAL = {J. Amer. Math. Soc.},
  FJOURNAL = {Journal of the American Mathematical Society},
    VOLUME = {26},
      YEAR = {2013},
    NUMBER = {4},
     PAGES = {1191--1207},
      ISSN = {0894-0347},
   MRCLASS = {53A30 (53A55)},
  MRNUMBER = {3073887},
MRREVIEWER = {Fr{\'e}d{\'e}ric Robert},
       DOI = {10.1090/S0894-0347-2013-00765-1},
       URL = {http://dx.doi.org/10.1090/S0894-0347-2013-00765-1},
}

@article {FeffermanHirachi2003,
    AUTHOR = {Fefferman, Charles and Hirachi, Kengo},
     TITLE = {Ambient metric construction of {$Q$}-curvature in conformal
              and {CR} geometries},
   JOURNAL = {Math. Res. Lett.},
  FJOURNAL = {Mathematical Research Letters},
    VOLUME = {10},
      YEAR = {2003},
    NUMBER = {5-6},
     PAGES = {819--831},
      ISSN = {1073-2780},
   MRCLASS = {53C20 (32V05 53A30 53A55)},
  MRNUMBER = {2025058},
MRREVIEWER = {A. Rod Gover},
       DOI = {10.4310/MRL.2003.v10.n6.a9},
       URL = {https://doi.org/10.4310/MRL.2003.v10.n6.a9},
}

@article {GoverGraham2005,
    AUTHOR = {Gover, A. Rod and Graham, C. Robin},
     TITLE = {C{R} invariant powers of the sub-{L}aplacian},
   JOURNAL = {J. Reine Angew. Math.},
  FJOURNAL = {Journal f\"ur die Reine und Angewandte Mathematik},
    VOLUME = {583},
      YEAR = {2005},
     PAGES = {1--27},
      ISSN = {0075-4102},
     CODEN = {JRMAA8},
   MRCLASS = {32V05 (53C15 58J60)},
  MRNUMBER = {2146851},
MRREVIEWER = {Andreas Cap},
       DOI = {10.1515/crll.2005.2005.583.1},
       URL = {http://dx.doi.org/10.1515/crll.2005.2005.583.1},
}

@article {GJMS1992,
    AUTHOR = {Graham, C. Robin and Jenne, Ralph and Mason, Lionel J. and Sparling, George A. J.},
     TITLE = {Conformally invariant powers of the {L}aplacian. {I}.
              {E}xistence},
   JOURNAL = {J. London Math. Soc. (2)},
  FJOURNAL = {Journal of the London Mathematical Society. Second Series},
    VOLUME = {46},
      YEAR = {1992},
    NUMBER = {3},
     PAGES = {557--565},
      ISSN = {0024-6107},
     CODEN = {JLMSAK},
   MRCLASS = {58G35 (53A30)},
  MRNUMBER = {MR1190438},
MRREVIEWER = {Michael G. Eastwood},
}

@article {GrahamZworski2003,
    AUTHOR = {Graham, C. Robin and Zworski, Maciej},
     TITLE = {Scattering matrix in conformal geometry},
   JOURNAL = {Invent. Math.},
  FJOURNAL = {Inventiones Mathematicae},
    VOLUME = {152},
      YEAR = {2003},
    NUMBER = {1},
     PAGES = {89--118},
      ISSN = {0020-9910},
   MRCLASS = {58J50},
  MRNUMBER = {1965361},
MRREVIEWER = {Andrew W. Hassell},
       DOI = {10.1007/s00222-002-0268-1},
       URL = {https://doi.org/10.1007/s00222-002-0268-1},
}

@book {Juhl2009,
    AUTHOR = {Juhl, Andreas},
     TITLE = {Families of conformally covariant differential operators,
              {$Q$}-curvature and holography},
    SERIES = {Progress in Mathematics},
    VOLUME = {275},
 PUBLISHER = {Birkh\"auser Verlag, Basel},
      YEAR = {2009},
     PAGES = {xiv+488},
      ISBN = {978-3-7643-9899-6},
   MRCLASS = {58J60 (11M36 35P25 53A30 53C21 53C25)},
  MRNUMBER = {2521913},
MRREVIEWER = {A. Rod Gover},
       DOI = {10.1007/978-3-7643-9900-9},
       URL = {http://dx.doi.org/10.1007/978-3-7643-9900-9},
}

@article {Matsumoto2013,
    AUTHOR = {Matsumoto, Yoshihiko},
     TITLE = {A {GJMS} construction for 2-tensors and the second variation
              of the total {$Q$}-curvature},
   JOURNAL = {Pacific J. Math.},
  FJOURNAL = {Pacific Journal of Mathematics},
    VOLUME = {262},
      YEAR = {2013},
    NUMBER = {2},
     PAGES = {437--455},
      ISSN = {0030-8730},
   MRCLASS = {58J05 (31C12 53A30 53A55)},
  MRNUMBER = {3069069},
MRREVIEWER = {Josef \v{S}ilhan},
       DOI = {10.2140/pjm.2013.262.437},
       URL = {https://doi.org/10.2140/pjm.2013.262.437},
}

@article {OvsienkoRedou2003,
    AUTHOR = {Ovsienko, V. and Redou, P.},
     TITLE = {Generalized transvectants-{R}ankin-{C}ohen brackets},
   JOURNAL = {Lett. Math. Phys.},
  FJOURNAL = {Letters in Mathematical Physics},
    VOLUME = {63},
      YEAR = {2003},
    NUMBER = {1},
     PAGES = {19--28},
      ISSN = {0377-9017},
   MRCLASS = {58J70 (53A30)},
  MRNUMBER = {1967533},
MRREVIEWER = {Sofiane Bouarroudj},
       DOI = {10.1023/A:1022956710255},
       URL = {https://doi.org/10.1023/A:1022956710255},
}

\end{document}